\documentclass[letterpaper, 11pt,  reqno]{amsart}
\usepackage[margin=1.2in,marginparwidth=2.5cm, marginparsep=0.5cm]{geometry}

\usepackage{amsmath,amssymb,amscd,amsthm,amsxtra, esint}
\usepackage{tikz}

\usepackage{microtype}
\usepackage{mhequ}
\usepackage{color}
\usepackage{xstring}
\usepackage{mathabx}

 \usepackage{comment}
 \usepackage{enumitem}
 \usepackage{slashed}

\usepackage{cases}%%%

\usepackage{orcidlink}

\makeatletter

\DeclareRobustCommand{\TitleEquation}[2]{\texorpdfstring{\StrLeft{\f@series}{1}[\@firstchar]$\if%
b\@firstchar\boldsymbol{#1}\else#1\fi$}{#2}}

\makeatother

\colorlet{darkblue}{blue!90!black}
\colorlet{darkred}{red!80!black}
\colorlet{darkgreen}{green!50!black}

\allowdisplaybreaks[2]

\definecolor{gr}{rgb}   {0.,   0.69,   0.23 }
\definecolor{bl}{rgb}   {0.,   0.5,   1. }
\definecolor{mg}{rgb}   {0.85,  0.,    0.85}
\definecolor{yl}{rgb}   {0.8,  0.7,   0.}
\definecolor{or}{rgb}  {0.7,0.2,0.2}

\usetikzlibrary{shapes.misc}
\usetikzlibrary{shapes.symbols}
\usetikzlibrary{decorations}
\usetikzlibrary{decorations.markings}

\tikzset{
	dot/.style={circle,fill=black,draw=black,inner sep=0pt,minimum size=0.5mm},
	>=stealth,
	}

\tikzset{
	dot2/.style={circle,fill=black,draw=black,inner sep=0pt,minimum size=0.2mm},
	>=stealth,
	}

\tikzset{
	ddot/.style={circle,fill=white,draw=black,inner sep=0pt,minimum size=0.8mm},
	>=stealth,
	}

\tikzset{decision/.style={ % requires library shapes.geometric
        draw,
        diamond,
        aspect=1.5
    }}

\tikzset{dia2/.style
={diamond,fill=white,draw=black,inner sep=0pt,minimum size=1mm},
	>=stealth,
	}

\tikzset{dia/.style
={star,fill=black,draw=black,inner sep=0pt,minimum size=1mm},
	>=stealth,
	}

\tikzset{dia/.style
={diamond,fill=black,draw=black,inner sep=0pt,minimum size=1.3mm},
	>=stealth,
	}

\makeatletter

\def\<#1>{\xusebox{#1}}
\makeatother

\newtheorem{theorem}{Theorem} [section]

\newtheorem{lemma}[theorem]{Lemma}
\newtheorem{proposition}[theorem]{Proposition}

\newtheoremstyle{myremark}% 〈name〉
{}% 〈Space above〉1
{}% 〈Space below 〉1
{}% 〈Body font〉
{}% 〈Indent amount〉2
{\bfseries}% 〈Theorem head font〉
{.}% 〈Punctuation after theorem head 〉
{ }% 〈Space after theorem head 〉3
{}%

\theoremstyle{myremark}
\newtheorem{remark}[theorem]{Remark}
\newtheorem*{ackno}{Acknowledgements}

\DeclareMathOperator*{\intt}{\int}

\DeclareMathOperator*{\supp}{supp}

\newcommand{\Ta}{\Theta}

\newcommand{\I}{\mathcal{I}}

\newcommand{\Z}{\mathbb{Z}}
\newcommand{\R}{\mathbb{R}}
\newcommand{\C}{\mathbb{C}}
\newcommand{\T}{\mathbb{T}}

\newcommand{\CC}{\mathcal{C}}

\newcommand{\D}{\mathcal{D}}

\let\Re=\undefined\DeclareMathOperator*{\Re}{Re}
\let\Im=\undefined\DeclareMathOperator*{\Im}{Im}

\let\P= \undefined
\newcommand{\P}{\mathbf{P}}

\newcommand{\Pii}{\mathbf{\Pi}}

\newcommand{\E}{\mathbb{E}}

\renewcommand{\L}{\mathcal{L}}

\newcommand{\F}{\mathcal{F}}

\newcommand{\al}{\alpha}
\newcommand{\be}{\beta}
\newcommand{\dl}{\delta}

\newcommand{\nb}{\nabla}

\newcommand{\Dl}{\Delta}
\newcommand{\eps}{\varepsilon}
\newcommand{\kk}{\kappa}
\newcommand{\g}{\gamma}

\newcommand{\ld}{\lambda}

\newcommand{\s}{\sigma}

\newcommand{\ft}{\widehat}

\newcommand{\wt}{\widetilde}

\newcommand{\dt}{\partial_t}

\newcommand{\ta}{\theta}

\renewcommand{\l}{\ell}

\renewcommand{\O}{\Omega}

\newcommand{\les}{\lesssim}

\newcommand{\jb}[1]
{\langle #1 \rangle}

\newcommand{\jbb}[1]
{[\hspace{-0.6mm}[ #1 ]\hspace{-0.6mm}]}

\newcommand{\ind}{\mathbf 1}

\newcommand{\too}{\longrightarrow}

\newcommand{\MM}{\mathcal{M}}

\newcommand{\N}{\mathbb{N}}

\renewcommand{\H}{\mathcal{H}}

\newcommand{\U}{\Theta}

\newcommand{\PP}{\mathbb{P}}

\renewcommand{\C}{\mathcal{C}}

\DeclareMathOperator{\Law}{Law}

\newcommand{\vertiii}[1]{{\left\vert\kern-0.25ex\left\vert\kern-0.25ex\left\vert #1 
    \right\vert\kern-0.25ex\right\vert\kern-0.25ex\right\vert}}

\setlist[enumerate]{label=\textup{(\roman*)}, align=left, leftmargin=0em, labelsep=0.4em, labelwidth=1.3em, itemindent=1.7em, itemsep=0.3em}

\newcommand{\wave}{\textup{wave}}

\numberwithin{equation}{section}
\numberwithin{theorem}{section}

\makeatletter
\@namedef{subjclassname@2020}{%
  \textup{2020} Mathematics Subject Classification}
\makeatother

\begin{document}
\baselineskip = 14pt

\title[Sine-Gordon model via stochastic quantization]
{A simple construction of the sine-Gordon model via stochastic quantization}

\author[M.~Gubinelli, M.~Hairer, T.~Oh, and Y.~Zine]
{Massimiliano Gubinelli\orcidlink{0000-0002-4014-2949}, Martin Hairer\orcidlink{0000-0002-2141-6561}, Tadahiro Oh\orcidlink{0000-0003-2313-1145},  and Younes Zine\orcidlink{0009-0001-7752-1205}}

\address{
Massimiliano Gubinelli\\
Mathematical Institute\\
University of Oxford\\
Andrew Wiles Building\\
Radcliffe Observatory Quarter\\
Woodstock Road\\
Oxford\\
OX2 6GG, 
 United Kingdom}

\email{gubinelli@maths.ox.ac.uk}

\address{
Martin Hairer\\
 \'Ecole Polytechnique F\'ed\'erale de Lausanne\\
1015 Lausanne\\ Switzerland, 
and 
Imperial College London, London SW7 2AZ, United Kingdom}

\email{martin.hairer@epfl.ch}

\address{
Tadahiro Oh, School of Mathematics\\
The University of Edinburgh\\
and The Maxwell Institute for the Mathematical Sciences\\
James Clerk Maxwell Building\\
The King's Buildings\\
Peter Guthrie Tait Road\\
Edinburgh\\ 
EH9 3FD\\
United Kingdom,
 and 
 School of Mathematics and Statistics, Beijing Institute of Technology, Beijing 100081,
China}

%\address{
%Tadahiro Oh, School of Mathematics\\
% School of Mathematics and Statistics, Beijing Institute of Technology, Beijing 100081,
%China\\
%and
%The University of Edinburgh\\
%and The Maxwell Institute for the Mathematical Sciences\\
%James Clerk Maxwell Building\\
%The King's Buildings\\
%Peter Guthrie Tait Road\\
%Edinburgh\\ 
%EH9 3FD\\
%United Kingdom} 
%
%
 
\email{hiro.oh@ed.ac.uk}

\address{
Younes Zine\\
 \'Ecole Polytechnique F\'ed\'erale de Lausanne\\
1015 Lausanne\\ Switzerland}

\email{younes.zine@epfl.ch}



\begin{abstract}
We present a simple PDE  construction of the sine-Gordon measure below the first threshold 
($\be^2 < 4\pi$),  in both the finite and infinite volume settings,  by studying the corresponding parabolic sine-Gordon model.
We also establish pathwise global well-posedness
of the hyperbolic sine-Gordon model in finite volume
for $\be^2 < 2\pi$.
\end{abstract}

%\date{\today}

\maketitle

\tableofcontents

\section{Introduction}
%\label{SEC:1}

\subsection{Sine-Gordon model}%\label{SUBSEC:1-1} 

The Euclidean sine-Gordon field theory (with unit mass) on the periodic box
 $\T_L^2 = (\R/2\pi L\Z)^2$, $L\ge 1$, 
is defined by 
the  sine-Gordon measure on 
$\D'(\T_L^2)$ with a  formal density:
\begin{align}
\begin{split}
d\rho_L^\be(u) & =``Z_L^{-1} \exp \bigg( \frac\g\be \int_{\T_L^2}\cos\big(\be u(x)\big)dx \\
& \hphantom{XXXXXX}
- \frac12 \int_{\T_L^2} u(x)^2 dx -  \frac12\int_{\T_L^2}|\nabla u(x)|^2dx \bigg) du",
\end{split}
\label{mes}
\end{align}
where 
$\be, \g \in \R\setminus \{0\}$
and $ Z_L =  Z(\be, \g,L) > 0$ is a normalization constant.
 In the current finite volume setting, the measure \eqref{mes} has been rigorously constructed in the mathematical physics literature for the full subcritical range $0 < \be^2 < 8\pi$ \cite{DH1, DH2, DH3};
 see also \cite{Fro, FP, BB}.
Our main goal in this  note is to present 
a simple construction of the sine-Gordon measure $\rho^\be$
via a stochastic quantization approach 
in the restricted range $0 < \be^2 < 4\pi$.

A  naive attempt 
to define the sine-Gordon measure $\rho^\be_L$ 
is to view it as a weighted Gaussian measure:
 \begin{align*}
 d \rho^\be_L(u) = ``  Z_L^{-1}  \exp \bigg( \frac\g\be \int_{\T^2_L}\cos\big(\be u(x)\big)dx \bigg) d\mu_L(u)", 
 \end{align*} 
where $\mu_L$ denotes the massive Gaussian free field on $\T^2_L$.
Namely, $\mu_L$ is a Gaussian measure on $L$-periodic distributions 
with the covariance operator $(1-\Dl)^{-1}$.
Unfortunately, $\mu_L$ on $\T^2_L$ is supported on 
distributions of negative regularity  \cite{GF}
and  thus we cannot make sense of $\cos (\be u)$ %the interaction potential $\int_{\T^2}\cos(\be u) dx$
for  general $u \in \supp(\mu_L)$.

Let us make this last observation  more precise. 
Fix $L \ge 1$. Given  $N \in \N$,  let  $\Pii_N = \Pii_{N}^L$  be a smooth frequency projector
onto the frequencies  $\{n\in\Z_L^2:|n|\les N\}$
defined by 
\begin{align}
\ft{ \Pii_{N} f}(n) =   \chi_N(n) \ft f (n), \quad n \in \Z_L^2
: = (\Z/L)^2, 
\label{chi}
\end{align}
where 
$\chi_N(n) = \chi(N^{-1}  n)$ for a fixed 
function
 $\chi \in C^\infty(\R^2; [0, 1])$ 
such that 
$\widecheck \chi\equiv 1$ 
on $\{x\in\R^2:|x|\leq \tfrac12\}$
and $\supp (\widecheck \chi) \subset \{x\in\R^2:|x|\leq 1\}$.
 Then, a variant of Proposition \ref{PROP:igmc} below
 (see \cite[Theorem 3.5]{HS}), 
 we have, for any test function $\phi \in C^\infty (\T^2_L)$, 
\begin{align*}
\E_{\mu_L} [  \jb{ \cos(\Pii_N u), \phi}^2 ] \les N^{-\frac{\be^2}{4\pi}} \too 0,
%\label{mes2b}
\end{align*}
as $N \to \infty$, 
where  $\jb{ \cdot, \cdot}$ denotes the usual duality pairing between distribution and test function. 
Thus,   we expect that the frequency-truncated sine-Gordon measure:
\begin{align*}
Z_{L,N}^{-1}  \exp 
\bigg( \frac{\g_L}{\be} \int_{\T^2_L}\cos\big(\be \Pii_N u(x)\big)dx \bigg) d\mu_L(u)
\end{align*}
converges weakly to the Gaussian measure $\mu_L$ as $N \to \infty$. 

 In order to circumvent this issue, 
 we
allow  the coupling constant $\g_L$
 to depend on $N \in \N$ and to diverge as $N \to \infty$.
 This procedure is referred to as {\it renormalization} in the literature on Euclidean quantum field theory 
 and singular stochastic PDEs; see for instance \cite{GF, Hairer}. 
For our purpose, define $\g_{L, N} = \g_{L, N}(\be)$,  $N \in \N$, by setting
 \begin{equation}
\g_{L, N}  = e^{\frac{\be^2}{2}\s_{L, N}}, 
\label{CN}
 \end{equation}
where 
$\s_{L, N}$ is defined by\footnote{Here, we use a Riemann sum approximation and the decay of $\chi$ in both physical and Fourier spaces.}
 \begin{equs}
 \s_{L, N}  &=  \E_{\mu_L}\big[\big(\Pii_{N}u(x)\big)^2\big] 
  = \frac 1{4\pi^2} \sum_{n\in\Z_L^2}\frac{\chi_N^2(n)}{\jb{ n }^2}\frac 1{L^2} 
   = \frac1{4\pi^2}\intt_{\R^2}
\frac{\chi_N(z)dz}{1 + |z|^2} + O(1)\\
& = \frac 1{2\pi}\log  N + O(1)
\label{sN}
 \end{equs}
as $N \to \infty$, 
independent of $x\in\T_L^2$, where $\jb{n} = (1+ |n|^2)^\frac 12$
and  the remainder $O(1)$ is bounded uniformly in the parameter $L \ge 1$;
see   \cite[Lemma 3.2]{HRW}
for the $L= 1$ case.
From \eqref{CN} and~\eqref{sN}, we have 
$%\begin{align*}
\g_{L, N} \sim N^\frac{\be^2}{4\pi} \to \infty
$ 
as $N \to \infty$, independently of $L\ge 1$.
Then, we define  the renormalized  truncated sine-Gordon measure
$\rho_{L, N}^\be$ by 
 \begin{align}
 d \rho_{L, N}^\be(u) = Z_{L, N}^{-1}  \exp \bigg( \frac {\g_{L, N}}{\be} \int_{\T_L^2}\cos\big(\be \Pii_{N} u(x)\big)dx \bigg) d\mu_L(u).
 \label{mes3}
 \end{align} 
Inheriting the regularity property
of the Gaussian free field $\mu_L$, 
the truncated sine-Gordon measure $\rho_{L, N}^\be$
is supported on $\CC^{-\dl}(\T_L^2)$
for any $\dl > 0$.
%where  $\CC^s(\T^2) = B^s_{\infty, \infty}(\T^2)$ denotes
% the H\"older-Besov space defined in Section~\ref{SEC:N}.

We now state our main results.
 
\begin{theorem}\label{THM:main}
Let   $0 < \be^2 < 4\pi$, $L>0$,  and $\dl > 0$.
% Given $N \in \N$, let $\rho^\be_N$ be the truncated sine-Gordon measure  as in~\eqref{mes3}. 
 Then, 
 there exists a subsequence $\{\rho_{L, N_j}^\be\}_{j \in \N}$ 
converging weakly to a limiting probability measure 
 $\rho^\be_L$ on $\C^{-\dl}(\T_L^2)$ 
 as $j \to \infty$.
 \end{theorem}

 As mentioned above, the construction of the sine-Gordon measure $\rho_L^\be$
 has been completed for the full subcritical range $0 < \be^2 < 8\pi$.
 See also \cite[Theorem 1.1]{ORSW} for the construction
 of the sine-Gordon measure $\rho_L^\be$
 for $0 < \be^2 < 4\pi$
 via a straightforward application of the variational approach~\cite{BG}. 
 Our main goal in this note
 is to present 
a simple proof of Theorem~\ref{THM:main} 
by studying  the  associated stochastic quantization 
equation 
for the sine-Gordon model
(namely, the parabolic sine-Gordon model~\eqref{SSG} below)
with 
techniques 
arising  from the study of singular stochastic PDEs.

Given $L \ge 1$, let $\rho^\be_L$ be  the sine-Gordon measure  on $\T^2_L$ constructed in Theorem \ref{THM:main}, which we view as a probability measure on  $\mathcal D'(\R^2)$, 
but supported on 
$2\pi L$-periodic distributions.
Then, by taking a large torus limit ($L \to \infty$) of the $2\pi L$-periodic sine-Gordon
measure $\rho^\be_L$, 
our approach also allows us to construct 
the sine-Gordon measure on $\R^2$.

 \begin{theorem}\label{THM:main2}
Let $0 < \be^2 < 4\pi$ and $\dl > 0$. 
% Given $N \in \N$, let $\rho^\be_N$ be the truncated sine-Gordon measure  as in~\eqref{mes3}. 
 Then, 
 there exists an increasing  sequence $\{L_j\}_{j \in \N}$
 of positive numbers such that 
the sequence  $\{\rho_{L_j}^\be\}_{j \in \N}$ 
of the $2\pi L_j$-periodic sine-Gordon measures
converges weakly to a limiting probability measure 
 $\rho^\be$ on $\mathcal{D}'(\R^2)$
 as $j \to \infty$.
 \end{theorem}

More precisely, 
the weak convergence of  $\{\rho_{L_j}^\be\}_{j \in \N}$  takes places 
as measures on 
the weighted H\"older-Besov space
$\CC^{-\dl}_{\ld, M} (\R^2)$ defined in \eqref{A3}
for some  $\ld, M \gg 1$.

\subsection{Parabolic sine-Gordon model}

Let $L\ge 1$ and consider 
the following stochastic nonlinear heat equation (SNLH) on $\T_L^2$
with a sine nonlinearity:
\begin{align}
%\begin{cases}
\dt u_L  + (1-\Dl)  u_L   + \g \sin(\be u_L) = \sqrt{2}\xi_L, 
%u_|{t=0}  = u_0 , 
%\end{cases}
\qquad (t, x) \in \R_+\times\T_L^2,
\label{SSG}
\end{align}
where  $u_L$ is a real-valued unknown and $\xi_L$ denotes a space-time white noise on $\R_+ \times \T_L^2$. Namely, $\xi_L$ is a distribution-valued Gaussian process with covariance given by
\begin{align}
\E\big[ \jb{\xi_L, \phi} \jb{\xi_L, \psi}\big] = \jb{\phi, \psi}_{L^2_{t,x}},
\label{noise}
\end{align}
for any $\phi, \psi \in \mathcal D(\R_+ \times \T^2_L)$. Here, $\jb{\cdot, \cdot}$ denotes the $\D(\R_+ \times \T^2_L)-\D'(\R_+ \times \T^2_L)$ duality pairing.

% with the space-time covariance formally given by
%\[ \E \big[\xi(x_1, t_1 ) \xi(x_2, t_2)\big] =  \dl(x_1- x_2) \dl(t_1-t_2).\]
The dynamical model \eqref{SSG}  corresponds to the so-called  {\it stochastic quantization}
\cite{PW}
of the quantum sine-Gordon model represented by the measure $\rho_L^\be$ in \eqref{mes}
in the sense that the sine-Gordon measure $\rho_L^\be$
is formally invariant under the dynamics generated by \eqref{SSG}.
For this reason, we refer
to \eqref{SSG}
as the {\it parabolic sine-Gordon model} in the following.

In \cite{HS, CHS}, 
the second author with Chandra and Shen
observed that 
the difficulty of the problem depends sensitively on the value of $\be^2 > 0$
(see, for example, Proposition~\ref{PROP:igmc})
and showed that 
 there is an infinite number of thresholds:
$\be^2 = \frac{j}{j+1}8\pi$, $j\in\N$, 
where one encounters new divergent stochastic objects, requiring further renormalizations.
By adapting  the theory of regularity structures~\cite{Hairer}
to the  non-polynomial setting,
they 
proved local well-posedness of~\eqref{SSG}  (for sufficiently smooth initial data)
in the full subcritical range $0<\be^2<8\pi$.

Let us describe  how to study the dynamical problem \eqref{SSG}, 
at least below the first threshold $0 < \be^2 < 4\pi$. % namely below the first threshold.
Due to the roughness of the space-time white noise $\xi$, 
a solution $u(t)$ for any fixed $t > 0$ is merely a distribution and 
thus we cannot make sense of the nonlinearity $\sin(\be u)$ in~\eqref{SSG}. 
We  overcome this issue by 
transferring  the renormalization  at the level of the measure 
in  \eqref{mes3}
 to a renormalization  at the level of the dynamical problem.
Given $N \in \N$, consider the following 
 frequency-truncated equation:
\begin{align}
 \dt u_{L, N}  + (1-\Dl)  u_{L, N} 
+ \g_{L, N} \Pii_{N} \sin (\be \Pii_{N} u_{L, N})  = \sqrt{2} \xi_L , 
\label{rSSG1}
\end{align}
where $\g_{L, N}$ is as in \eqref{CN}.
Then, 
using the first order expansion (see~\eqref{expa} below)
along with the basic results in Section~\ref{SEC:N}, 
one can easily show
 that there exist an almost surely positive random time $T>0$ and a non-trivial distribution 
$u_L \in C([0,T]; \CC^{-\dl}(\T_L^2))$, $\dl > 0$, such that 
$u_{L, N}$ converges 
in probability 
to $u_L$
in 
$C([0,T]; \CC^{-\dl}(\T_L^2))$ as $N \to \infty$.
The limit $u_L$ is then interpreted as the solution to the (formal) renormalized version of \eqref{SSG}:
\begin{align}
\dt u_L +(1-\Dl) u_L   +  \infty\cdot  \sin(\be u_L) = \sqrt{2}\xi_L.
\label{rSSG2}
\end{align}
This is what we mean by pathwise local well-posedness
of the parabolic sine-Gordon model~\eqref{rSSG2}.\footnote{Strictly speaking, 
in studying pathwise well-posedness, 
one normally considers the nonlinearity
$\g_{L, N} \sin (\be \Pii_{N} u_{L, N})$
with the truncated noise 
 $\Pii_{N} \xi_L$; see \eqref{WW3}.
In our current discussion, however, 
we consider~\eqref{rSSG1}
since the truncated sine-Gordon measure $\rho^\be_{L, N}$ in \eqref{mes3}
is invariant under the dynamics \eqref{rSSG1}.}

%\medskip

We briefly outline the proof of 
Theorem~\ref{THM:main}
whose details are presented in Section~\ref{SEC:proof}.
We consider 
 \eqref{rSSG1}
 with the truncated sine-Gordon measure initial data, namely
$\Law(u_{L, N}(0)) = \rho_{L, N}^\be$, 
where $\Law(X)$ denotes the law of a random variable $X$.
Hereafter, we assume that the space-time white noise $\xi_L$
in \eqref{rSSG1} is independent
of $u_{L, N}(0)$.
%(and the limiting sine-Gordon measure $\rho^\be$ once constructed).
Then, 
a standard argument yields  the following invariance result,
see for example 
\cite[Subsection 5.2]{ORW}
for details
in the context of  SNLH 
with an exponential nonlinearity.

\begin{lemma}\label{LEM:inv1}
Let $L \ge 1$ and $N \in \N$.
Then, 
the truncated sine-Gordon measure
 $\rho^\be_{L, N}$
 in~\eqref{mes3}
 is  invariant under the truncated dynamics \eqref{rSSG1}.\qed
\end{lemma}

Our main strategy for proving Theorem \ref{THM:main}
is to study the truncated equation \eqref{rSSG1}
and obtain uniform (in~$N$) pathwise a priori bound
on the solution $u_{L, N}$ 
constructed in Lemma~\ref{LEM:inv1}
(Proposition \ref{PROP:key}).
Since we only consider the range $0 < \be^2 < 4\pi$, 
we  proceed with the  first order expansion as in \cite{McKean, BO96,DPD}:
\begin{align}
u_{L, N} = \Psi_L + v_{L, N}, 
\label{expa}
\end{align}
 where $\Psi_L$ is the so-called stochastic convolution, namely the unique stationary
 solution to
\begin{equ}
(\dt + 1 - \Dl) \Psi_L  = \sqrt{2}\xi_L\;.
\label{Psi0}
\end{equ}

By substituting  \eqref{expa} into \eqref{rSSG1}, we see  that $v_N = u_N - \Psi$ satisfies
\begin{align}
 \dt v_{L, N} + (1 -\Dl)  v_{L, N}
+ \g_{L, N} \Pii_N  \sin \big(\be( v_{L, N} + \Psi_{L, N} )\big) = 0\;,
\label{v0}
\end{align}
where
 $\Psi_{L, N} = \Pii_N \Psi_L$
 with $\Pii_N$  as in \eqref{chi}. 
By rewriting \eqref{v0} in its  mild formulation, we have  
\begin{align}
v_{L, N} (t)= e^{t (\Dl - 1)} \big( u_{L, N}(0) - \Psi_{L}(0)\big)  
- 
\Pii_{N}  \I \big(  \Im (f( v_{L, N}) \Theta_{L, N})\big) (t), 
\label{v1}
\end{align} 
where $f(v) = e^{i \be v}$, 
 $\I$ denotes  the heat Duhamel integral operator 
defined by 
\begin{align}
\I(F) (t) = \int_0^t e^{(t-t') (\Dl-1)} F(t')dt',
\label{duha} 
\end{align}
and $\Theta_{L, N}$ is the so-called {\it imaginary Gaussian multiplicative chaos} defined by
\begin{align}
\Theta_{L, N} = \g_{L, N}  e^{ i \be \Psi_{L, N}}
= e^{\frac{\be^2}{2}\s_{L, N}}
e^{ i \be \Psi_{L, N}}.
\label{igmc}
\end{align}
See Proposition~\ref{PROP:igmc}
for the (limiting) regularity property of $\Ta_{L, N}$.
 A key ingredient in the proof of Theorem~\ref{THM:main}
is a deterministic global-in-time a priori bound on solutions to 
the associated  nonlinear heat equation;
 see Proposition~\ref{PROP:key} below.
 Then, 
tightness of $\{\rho_{L, N}^\be\}_{N \in \N}$
follows 
 from applying such a global-in-time a priori bound
to the solution $v_N$ in \eqref{v1}
with   $ \Law(u_{L, N}(0)) = \rho_{L, N}^\be$
and combining it with the invariance
of $\Law( v_{L, N}(t) + \Psi_L(t))$.
See Section~\ref{SEC:proof} for details.

The proof of Theorem \ref{THM:main2} then follows from a refinement of the arguments leading to Theorem \ref{THM:main} and is presented in Section~\ref{SEC:R2}.

\begin{remark}
\begin{enumerate}
\item 
In recent years, there have been several works \cite{AK, GH2, AK2}
on the PDE construction of the $\Phi^4_3$-measures
on both $\T^3$ and $\R^3$, 
arising in Euclidean quantum field theory;
see also \cite{MW, OTWZ}.
While the contracting (= defocusing)  nature
of the interaction potential played an essential role in the aforementioned works, 
the sine-Gordon model is {\it not} contracting.
We instead make use of the fact that 
the nonlinearity $\sin(\be u)$ in 
 \eqref{rSSG2}  is bounded and Lipschitz;
see \eqref{X3} below.
See also recent preprints \cite{CFW, SZZ, BC}.
In particular, 
 in \cite[Corollary 3]{CFW}, Chandra, Feltes,  and Weber proved pathwise 
 global well-posedness for the parabolic sine-Gordon model~\eqref{rSSG2} 
% (with $L^\infty$-initial data)
below  the second threshold ($0 < \be^2 < \frac{16}{3}\pi$)
 and also deduced from their bounds the existence of an invariant measure for~\eqref{rSSG2} 
 by  the Bogoliubov-Krylov argument. 
In a very recent preprint \cite{BC}, 
Bringmann and Cao extended 
global well-posedness 
of~\eqref{rSSG2} 
to the third threshold  ($0 < \be^2 < 6 \pi$).
We point out that the results
\cite{CFW, SZZ, BC}
are restricted to the periodic setting.

\item In a recent preprint \cite{GM}, 
the first author with Meyer
constructed the sine-Gordon measure  in the infinite-volume setting 
(namely on $\R^2$) 
for $0 < \be^2 < 6\pi$
via 
the forward-backward stochastic differential equation approach;
see also a recent work \cite{Bara} on the range $0 < \be^2 < 4\pi$.
It is interesting to see how the third threshold
$\be^2 = 6\pi$ appears
in the recent (very independent)  works \cite{BB, GM, BC}.

\end{enumerate}
\end{remark}

 \begin{remark}
\begin{enumerate}
\item Proposition~\ref{PROP:key}
 yields  a pathwise global-in-time
a priori bound for solutions to the parabolic sine-Gordon model
\eqref{rSSG2}
and thus its pathwise global well-posedness
in the range $0 < \be^2 < 4 \pi$.
This provides an alternative, simpler proof
of the pathwise well-posedness result \cite[Corollary~3]{CFW}
in our much simpler setting. % (i.e.~below the first threshold $0 < \be^2 < 4\pi$).

\item 
Consider 
the following
hyperbolic sine-Gordon model on $\T^2 = (\R/2\pi \Z)^2$:
\begin{align}
\dt^2 u + \dt u + (1- \Dl)  u   +  \g \sin(\be u) = \sqrt{2}\xi, 
\label{WW1}
\end{align}
 corresponding to the 
so-called ``canonical" stochastic quantization~\cite{RSS} 
of the sine-Gordon model.
In~\cite{ORSW}, 
for the range $0 < \be^2 < 2\pi$, 
the authors constructed invariant Gibbs dynamics for~\eqref{WW1}, 
 based on Bourgain's invariant measure argument
\cite{BO94, BO96}.
A straightforward adaptation of the proof
of Proposition~\ref{PROP:key}
yields
a pathwise  global-in-time a priori bound on solutions to \eqref{WW1}
for the range $0 < \be^2 < 2\pi$, 
providing its pathwise global well-posedness
in the same range of $\be^2$
(Theorem \ref{THM:GWP}).
However, due to the lack of smoothing on (homogeneous)
linear solutions for the (damped) wave equation, this argument does not 
easily yield a PDE construction of the sine-Gordon measure $\rho^\be_{L = 1}$.
See Appendix~\ref{SEC:wave}
for a further discussion.
\end{enumerate}
\end{remark}

\section{Notations and preliminary results}
\label{SEC:N}

\subsection{Notations}\label{SUBSEC:N1} 
 We write $ A \les B $ to denote an estimate of the form $ A \leq CB $
 for some $C > 0$. 
We write  $ A \sim B $ to denote $ A \les B $ and $ B \les A $.
We may use subscripts to denote dependence on external parameters; for example,
 $A \les_\ta B$ means $A \leq C(\ta) B$.

Let $L \ge 1$. For $x \in \T^2_L \cong [-L\pi, L\pi)^2$, 
we set
\[ |x |_{L} = \min_{m \in \Z^2} |x  + 2\pi L m|_{\R^2},\]
where $|\cdot |_{\R^2}$ denotes the standard Euclidean distance on $\R^2$.
%When there is no confusion, 
%we simply use $|\cdot|$ for $|\cdot|_{\T^2}$ and $|\cdot |_{\R^2}$.

%By the (assumed)
%independence of 
%  the truncated sine-Gordon 
%measures $\rho_{L, N}^\be$ in \eqref{mes3}
%and (the distribution of) the  space-time white noise $\xi_L$ on $\R\times \T^2_L$
%(which is also viewed as the  $2\pi L$-periodic (in space)
%space-time white noise  on $\R\times \R^2$),  
%we can write the probability space $\O$ as
%\begin{equation}
%\O = \O_1 \times \O_2.
%\label{decomp1}
%\end{equation}
%such that, for each $L \ge 1$,  the random initial data $u_{L,  N}(0)$
%in \eqref{rSSG1} with 
%$\Law(u_{L, N}(0)) = \rho_{L, N}^\be$
%depends only on $\o_1 \in \O_1$, 
% while the $2\pi L$-periodic space-time white noise $\xi_L$
%in  \eqref{rSSG1} and~\eqref{Psi0}
%(and hence the stochastic convolution $\Psi_L$)
% depends only on $\o_2 \in \O_2$. 
% In view of \eqref{decomp1}, we also write the underlying probability measure 
% $\PP$ on $\O$ as 
% \begin{align}
% \PP = \PP_1 \otimes \PP_2,
% \label{decomp2}
% \end{align}
%  where $\PP_j$ is the marginal probability measure on 
% $\O_j$, $j = 1, 2$.
%

Fix $L\ge 1$. We define the Fourier transform $\ft f(n)$ (or $\F_{\T^2}(f)(n)$ depending on the context) of a function $f$ on $\T^2_L$
by 
\begin{align}
\ft f(n) %= \F_{\T^2_L} (f)(n)
= \frac{1}{2\pi} \int_{\T^2_L} f(x) e^{-i n \cdot x} dx, \quad n \in \Z^2_L, 
\label{FT1}
\end{align}
with the associated Fourier series expansion:
\begin{align*}
f(x) = \frac{1}{2\pi L^2} \sum_{n \in \Z_L^2} \ft f(n)e^{i n \cdot x}.
%\label{FT2}
\end{align*}
Note that we have
\begin{align}
\ft {f*_Lg}(n) = 2\pi  \ft f(n) \ft g(n)
%\ft {fg}(n) = \frac 1{2\pi L} \sum_{k \in \Z^2_L} \ft f(k) \ft g (n-k).
\label{FT3}
\end{align}
for any $n \in \Z_L^2$, where $*_L$ denotes the convolution on $\T^2_L$
defined by 
\[ f*_Lg(x) = \int_{\T^2_L}f(x - y) g(y) dy .\]

Given a function $f$ on $\R^2$,
we define its Fourier transform 
 $\F(f)$ (or $\F_{\R^2}(f)$ depending on the context)  by 
\begin{align}
 \F(f) (\xi) = \frac 1{2\pi}\int_{\R^2} f(x) e^{- i \xi\cdot  x} dx.
 \label{FT3a}
\end{align}
With this convention, we have
\begin{align}
\begin{split}
f(x)  & = \frac 1{2\pi} \int_{\R^2} \F(f)(\xi) e^{i \xi \cdot x} d\xi,\\
\F(f*g)(\xi)
& = 2\pi
\F(f)(\xi) \F(g)(\xi).
\end{split}
\label{FT4}
\end{align}

In the following, 
we  introduce 
 the Littlewood-Paley frequency projectors
for 
  $2\pi L$-periodic functions, $L \ge 1$
  via
the Poisson summation formula
so that they are consistent for different values of  $L \ge 1$.
Let $F$ be a continuous function on $\R^2$ 
such that (i)~there exists $\dl > 0$ such that 
\[|F(x)| \les (1+|x|)^{-2-\dl}\]
for any $x \in \R^2$, 
and (ii)~its Fourier transform restricted to $\Z^2_L$ satisfies 
\[ \sum_{n \in \Z^2_L} |\F(F)(n)| < \infty.\]
Define a function $f$ on $\T_L^2 \cong [-\pi L , \pi L)^2$ by setting
\begin{align}
 f(x) = \sum_{m \in \Z^2} F(x + 2\pi mL).
 \label{Poi1y}
\end{align}
Then, from \eqref{FT1}, and \eqref{FT3a}, 
we have 
\begin{align}
 \ft f (n) =  \int_{\T^2_L} \sum_{m \in \Z^2} F(x + 2\pi m L) \frac{e^{-  i n\cdot  x}}{2\pi }dx
=  \F (F) (n)
\label{Poi1}
\end{align}
for any $n \in \Z_L^2$.
In particular, 
the Poisson summation formula
on $\T^2_L$ with our convention
reads as follows:
\begin{align}
\frac{1}{2\pi L^2}\sum_{n \in \Z^2_L} \F (F)(n) e^L_n(x) 
= 
\frac{1}{2\pi L^2} \sum_{n \in \Z^2_L} \ft f(n) e^L_n(x) 
= 
 \sum_{m \in \Z^2} F(x + 2\pi m L).
%
%\big\| \jb{\nb}^\al K_t^{L}\big\|_{L^1(\T^2_L)}
%= \bigg\|
%\sum_{k \in \Z^2}
%\big(\jb{\nb}^\al K_t^{\R^2}\big)(x + 2\pi L k) 
%\bigg\|_{L^1(\T^2_L)}
%\les_\al \min (1, t)^{-\frac \al 2}, 
\label{Poi0}
\end{align}
See
\cite[Theorem 3.2.8]{Gra1}.

%Given $k \in \N \cup \{0\}$, 
%let $\Phi_k$ be the kernel of the Littlewood-Paley projector (on $\R$)
%onto  frequencies  $\{ |\xi|\sim 2^k\}$ for $k \in \N$
%and $\{ |\xi|\les 1\}$ for $k = 0$.

Let $\eta:\R \to [0, 1]$ be a smooth  bump function supported on $[-\frac{8}{5}, \frac{8}{5}]$ 
and $\eta\equiv 1$ on $\big[-\frac 54, \frac 54\big]$.
For $\xi \in \R^2$, we set 
\begin{align*}
\ft \Phi_0(\xi) = \eta(|\xi|)\quad \text{and}\quad 
\ft \Phi_{k}(\xi) = \eta\Big(\frac{|\xi|}{2^k}\Big)-\eta\Big(\frac{|\xi|}{2^{k-1}}\Big)
%\label{Poi1x}
\end{align*}
for $k \in \N$.
Defining $\Phi$ by 
$\ft \Phi(\xi) = \eta(|\xi|)-\eta(2^{-1} |\xi|)$, we have
\begin{align*}
 \Phi_k(x) =  2^{2k} \Phi(2^k x).
%\label{Poi1xx}
\end{align*}
Note that $\sum_{k = 0}^\infty \ft \Phi_k (\xi) = 1$
 for any $\xi \in \R^2$. 
We then define the Littlewood-Paley projector $\P_k$ on $\R^2$
by 
\begin{equ}
 \P_{k} f  =  \frac 1{2\pi} \Phi_k*f.
\label{Poi1z}
\end{equ}
for 
 $k \in \Z_{\ge 0} := \N \cup\{ 0\}$; see \eqref{FT4}.

Let $L \ge 1$.
For $k \in \Z_{\ge 0} $, define $\Psi_k^L$ on $\T_L^2$ by
\begin{align}
  \Psi_k^L (x)= \sum_{m \in \Z^2} \Phi_k(x + 2\pi mL).
  \label{Poi1a}
\end{align}
From \eqref{Poi1y} and  \eqref{Poi1}, we have
\begin{align}
\ft \Psi_k^L(n) = \F(\Phi_k)(n), \quad n \in \Z_L^2.
\label{Poi2}
\end{align}
We then define the Littlewood-Paley projector $\P_k^L$ 
(for functions on $\T^2_L$) 
onto the frequencies $\{n \in \Z^2_L : |n|\sim 2^k\}$ for $k \in \N$
and $\{n \in \Z^2_L : |n|\les 1\}$ for $k = 0$
by setting
\begin{align}
 \P_{k}^L f  = \frac{1}{2\pi} \Psi_k^L*_L f.
 \label{Poi2x}
\end{align}
Then, from 
\eqref{FT3} and \eqref{Poi2}, we have 
\[ \ft {\P_k^L f} (n)
= 
\ft \Psi_k^L(n)\ft f (n)
= \F(\Phi_k)(n)
\ft f (n), \quad n \in \Z_L^2.\]
Given a $2\pi L$-periodic function $f$ on $\R^2$, 
it follows 
from \eqref{Poi2x}, 
\eqref{Poi1a}, and \eqref{Poi1z} that 
\begin{align}
\P^L_k f (x) = \P_k f (x)
\label{Poi3}
\end{align}
for any $x \in \R^2$, where we view the left-hand side
as a $2\pi L$-periodic function on $\R^2$.

Given  $s \in \R $ and $1 \le p  \le \infty$,
the $L^p$-based Sobolev space $W^{s,p}(\T_L^2)$ is defined 
by  the norm:
\begin{align*}
\|f\|_{W^{s,p}(\T_L^2)} = \| \jb \nb ^s f   \|_{L^p(\T_L^2)} = \big\| \F_{\T^2_L}^{-1} \big[ \jb n ^s \ft f (n) \big]\! \big\|_{L^p(\T_L^2)}.
\end{align*}
When $p = 2$, we set $H^s(\T_L^2) = W^{s, 2}(\T_L^2)$.
 Given  $s \in \R $ and $1 \le p, q  \le \infty$,
the Besov space $B^{s}_{p,q}(\T_L^2)$ is defined as the closure of 
$C^\infty(\T_L^2)$  under the norm:
\begin{align}
\|f\|_{B^{s}_{p,q}(\T_L^2)} =  \Big\| 2^{ks} \|\P_k^L f \|_{L^p_x(\T_L^2)} \Big\|_{\l^q_k(\Z_{\ge 0})}, 
\label{B1}
\end{align}
where 
$\P_k^L$ is as in  \eqref{Poi2x}.
When $p = q = \infty$, we set  $\CC^s(\T_L^2) = B^s_{\infty, \infty}(\T_L^2)$
to denote the H\"older-Besov space. 
We define   $W^{s,p}(\R^2)$, $B^{s}_{p,q}(\R^2)$,  and $\CC^s(\R^2)$
in an analogous manner.
In particular, $B^s_{p, q}(\R^2)$ is defined as the closure of $C^\infty_c(\R^2)$
under the norm:
\begin{align}
\|f\|_{B^s_{p, q}(\R^2)} =  \Big\| 2^{ks} \|\P_k f \|_{L^p_x(\R^2)} \Big\|_{\l^q_k(\Z_{\ge 0})}, 
\label{B2}
\end{align}
where $\P_k$ is as in \eqref{Poi1z}.
Then, 
from \eqref{Poi3}, \eqref{B1}, and \eqref{B2}, 
we have
\begin{align}
\|f\|_{\CC^s(\T^2_L)} 
= \|f\|_{\CC^s(\R^2)}\label{B3}
\end{align}
for  a $2\pi L$-periodic function $f$ on $\R^2$.

%
%Suppose that a function $f$ on $\R^2$ is supported on 
%$ [-\pi L , \pi L)^2\cong  \T_L^2$.
%Then, from 
%\eqref{Poi2x}, \eqref{Poi1a}, and \eqref{Poi1z}, we have
%\begin{align*}
%\P_k^L f(x) = 
%\end{align*}

When it is clear from context, we may suppress the spatial domain
$\T^2_L$ or $\R^2$.
We also use shorthand notations such as
 $C_T\CC^s_x= C_T\CC^s_x(\T^2_L) = C([0, T]; \CC^s(\T_L^2))$.

\subsection{Deterministic estimates}
%\label{SUBSEC:N2} 

In the following, we
state basic deterministic estimates.
We first recall the Schauder estimate.
%which follows from Young's inequality.
% see  \cite[Lemma A.7]{GIP}. or \cite[Section 3]{MW2} for a proof. 

\begin{lemma}\label{LEM:Schauder}
Let  $s, s_1, s_2 \in \R$ with $s_2 \ge s_1$
and  $\mathcal M = \T^2_L$ or $\R^2$. Then, 
given $\al \ge 0$, 
we have
\begin{align*}
% \big\| |\nb|^\al e^{t\Dl} f \big\|_{\CC^{s}(\MM)} & \les t^{- \frac{\al}{2}} \|f\|_{\CC^{s}(\MM)}, 
% \label{Sch1}\\
\|e^{t\Dl} f \|_{\CC^{s_2}(\MM)} & \les \min(1,t)^{\frac{s_1 - s_2}{2}} \|f\|_{\CC^{s_1}(\MM)}
%\label{Sch2}
\end{align*}
for any $t > 0$, 
where the implicit constants are independent of $ L \ge 1$
when $\MM = \T^2_L$.\qed
\end{lemma}

%\noi
%The
%
%
% Note that, in the periodic setting (i.e.~on $\T^2_L$), 
%we do not have any decay (as $t \to \infty$) due to the zeroth Fourier  mode.
%

%\begin{proof}
%Let $K_t^{\R^2}$ and $K_t^{L}$ be the convolution kernels 
%of  the heat semigroup $ e^{t\Dl}$ on $\R^2$ and $\T^2_L$, respectively.
%When  $\MM = \R^2$, 
%the Schauder estimate %\eqref{Sch1} and 
%\eqref{Sch2} follows from 
%Young's inequality and 
%\begin{align}
%%\big\| |\nb|^\al K_t^{\R^2}\big\|_{L^1(\R^2)}
%%\sim_\al t^{-\frac \al 2}
%%\quad \text{and}\quad 
%\big\| \jb{\nb}^\al K_t^{\R^2}\big\|_{L^1(\R^2)}
%\sim_\al \min (1, t)^{-\frac \al 2}
%\label{Sch3}
%\end{align}
%for any $\al \ge 0$.
%From \eqref{FT3},  we have 
%$\ft K^{L}_t(n) = (2\pi L)^{-1}e^{-t |n|^2}
%=L^{-1}\ft K^{\R^2}_t(n)$ for each $n \in \Z_L^2$.
%Then, from the Poisson summation formula
%\eqref{Poi0}
%and \eqref{Sch3}, we have 
%\begin{align*}
%\big\| \jb{\nb}^\al K_t^{L}\big\|_{L^1(\T^2_L)}
%= \bigg\|
%\sum_{k \in \Z^2}
%\big(\jb{\nb}^\al K_t^{\R^2}\big)(x + 2\pi L k) 
%\bigg\|_{L^1(\T^2_L)}
%\les_\al \min (1, t)^{-\frac \al 2}, 
%\end{align*}
%uniformly in $L \ge 1$, 
%from which  \eqref{Sch2} follows via Young's inequality on $\T^2_L$.
%\end{proof}
%

The following product estimate is due to Bony \cite{Bony};
 see  \cite[Lemma 2.1]{GIP} or \cite[Section 3]{MW2} for a proof.

\begin{lemma}\label{LEM:prod}
Let $L\ge 1$ and $s_1 < 0 < s_2$ with $s_1 + s_2 >0$. 
Let $\mathcal M = \T^2_L$ or $\R^2$. Then, we have
\begin{equ}
\|fg \|_{\CC^{s_1}(\mathcal M)} \les \| f \|_{\CC^{s_1}(\mathcal M)} 
\| g\|_{\CC^{s_2}(\mathcal M)}, 
\end{equ}
where the implicit constant is independent of $L \ge 1$
when $\MM = \T^2_L$.\qed
\end{lemma}

Lastly, we state the following elementary lemma.

\begin{lemma}\label{LEM:elem}
Let  $ s_1, s_2 < 1$. Then, 
for any $t  > 0$, 
we have \begin{align*}
\int_{0}^t (t-t')^{-s_1} (t')^{-s_2} dt' 
=  B(1-s_1, 1-s_2) \, t^{1-s_1 - s_2}.
\end{align*}
Here,  $B(z_1, z_2)$ is the beta function at $(z_1, z_2)$
which is finite for $\Re z_1, \Re z_2 > 0$.\qed
\end{lemma}

\subsection{Tools from stochastic analysis}
%\label{SEC:proba}

In this subsection,  
our main goal is to establish 
the following two propositions on 
regularity properties 
of the 
truncated stochastic convolution 
 $\Psi_{L, N} = \Pii_N \Psi_L$, 
 where 
 $\Psi_L$ 
 and 
 $\Pii_N$ 
are as in  
 \eqref{Psi0} and 
  \eqref{chi}, respectively, 
 and 
the truncated  imaginary Gaussian multiplicative chaos
$\Ta_{L, N}$
defined in \eqref{igmc}
 for $L \ge 1$ and $N \in \N$.
% see Propositions \ref{PROP:psi}
% and \ref{PROP:igmc}.

 \begin{proposition}\label{PROP:psi}
Let $L\ge 1$. Fix $T>0$, $\dl > 0$,  and $1 \le p < \infty$. %Then the following holds:
\begin{enumerate}
\item The sequence
$ \{\Psi_{L, N} \}_{N \in \N}$ is a Cauchy sequence in $L^p(\O;C([0,T]; \CC^{-\dl}(\T_L^2)))$.
Moreover, we have 
\begin{align}
\sup_{N \in \N\cup\{\infty\}}  \E\Big[ \|\Psi_{L, N}\|_{C_T \CC_x^{-\dl }(\T^2_L)}^p \Big] 
\les_{T, p, \dl} L^2,
\label{GX0a}
\end{align}
where the implicit constant is independent of $L \ge 1$. Here, $\Psi_{L, \infty}$ denotes the limit of $\Psi_{L, N}$
and is equal to $\Psi_L$ defined in \eqref{Psi0}.

\item Let $B \subset \R^2$ be a disk of radius at least $1$ and let $\chi_B$ be a smooth function supported on B such that $\|\chi_B\|_{W^{1,\infty}(\R^2)} \les 1$. The sequence
$ \{\chi_B \Psi_{L, N} \}_{N \in \N}$ is a Cauchy sequence in $L^p(\O;C([0,T]; \CC^{-\dl}(\R^2)))$. 
Moreover, we have 
\begin{align}
\sup_{L \ge 1} \sup_{N \in \N\cup\{\infty\}}  \E\Big[ \| \chi_B \Psi_{L, N}\|_{C_T \CC_x^{-\dl}(\R^2)}^p \Big] 
\les_{T, p, \dl} \min(|B|,L^2),
\label{GX0b}
\end{align}
where the  implicit constant is independent of $B$ and $L$.
\end{enumerate}
\end{proposition}

 \begin{proposition}\label{PROP:igmc}
Let $L\ge 1$ and $0 < \be^2 < 4\pi$. Fix $T>0$, $\al > \frac{\be^2}{4\pi}$, and $1 \le p < \infty$. 
\begin{enumerate}
\item The sequence
$\{ \Theta_{L, N} \}_{N\in \N}$ is a Cauchy sequence  in $L^p(\O;C([0,T]; \CC^{-\al}(\T_L^2)))$.
Moreover, we have 
\begin{align}
\sup_{N \in \N\cup \{\infty\}}\E\Big[ \| \Theta_{L, N}\|_{C_T \CC_x^{-\al }(\T^2_L)}^p \Big] 
\les_{T, p, \al} L^{2},
\label{YY1}
\end{align}
where the implicit constant is independent of $L \ge 1$.
Here, $\Theta_{L, \infty}$ denotes the limit of $\Theta_{L, N}$.
\item Let $B \subset \R^2$ be a disk of radius at least $1$ and let $\chi_B$ be a smooth function supported on B such that $\|\chi_B\|_{W^{1,\infty}(\R^2)} \les 1$. The sequence
$ \{\chi_B \Theta_{L, N} \}_{N \in \N}$ is a Cauchy sequence in $L^p(\O;C([0,T]; \CC^{-\al}(\R^2)))$.
Moreover, we have 
\begin{align}
\sup_{L \ge 1} \sup_{N \in \N\cup\{\infty\}}  \E\Big[ \| \chi_B \Theta_{L, N}\|_{C_T \CC_x^{-\al}(\R^2)}^p \Big] 
\les_{T, p, \al} \min(|B|,L^2),
\label{YY2}
\end{align}
where the implicit constant is independent of $B$ and $L$.
\end{enumerate}
\end{proposition}

Proposition \ref{PROP:psi}\,(ii)
and Proposition \ref{PROP:igmc}\,(ii)
with a spatial localization by $\chi_B$
play an important role
in proving Theorem \ref{THM:main2}
in the infinite volume setting.
We note that, 
as for Proposition \ref{PROP:psi}
on 
the truncated stochastic convolution $\Psi_{L,N}$, 
Mourrat and Weber presented an analogous result 
(but in a slightly different norm);
see \cite[Theorem 5.1]{MW2}.
On the other hand, 
 Proposition \ref{PROP:igmc} 
on the  truncated  imaginary Gaussian multiplicative chaos
$\Ta_{L, N}$
 extends 
 the results in 
 \cite{HS, ORSW} to the current large torus setting,
 which requires some care
 in quantifying the dependence on $L$ and $|B|$.
For the sake of completeness, 
 we present proofs of Propositions~\ref{PROP:psi} and~\ref{PROP:igmc}
 in the following. 

%While these results follow from adapting the argument in \cite{MW2, HS, ORSW}, for the sake of completeness, we present details of the proofs of Propositions \ref{PROP:psi} and \ref{PROP:igmc} in the following.

We now introduce the setup for our argument. Our first goal is to introduce a suitable spatial truncation of the standard heat kernel. Let $K$ be the heat kernel on $\R_+ \times \R^2$, namely,
\begin{align*}
K(t,x) = \frac{1}{4\pi t} \exp\Big(- \frac{|x|_{\R^2}^2}{4t}\Big)
\end{align*}
for $(t,x) \in \R_+ \times \R^2$ so that with \eqref{FT4} we have
\[\F_{\R^2}\big(K(t,\cdot) * f\big)(\xi) = e^{-t|\xi|^2} \F_{\R^2}(f)(\xi)\]
for $(t,\xi) \in \R_+ \times \R^2$.

Let $L \ge 1$. The periodic heat kernel $K_L$ on $\R_+ \times \T^2_L$ is defined via the Fourier series
\[K_L(t,x) :=\frac{1}{4\pi^2 L^2} \sum_{n\in\Z_L^2}e^{-t |n|^2} e^{i n \cdot x} \]
for $(t,x) \in \R_+ \times \T^2_L$. By the Poisson formula \eqref{Poi0}, we have
\begin{align*}
K_L(t,x) = \sum_{m \in \Z^2} K(t, x + 2\pi m L)
\end{align*}
for $(t,x) \in \R_+ \times \T_L^2$.

Let $\phi \in C_c^{\infty}(\R^2; [0,1])$ be a smooth bump function such that $\phi \equiv 1$ on the disk $B(0,\frac12)$ and define the kernels $K^1_L$ and $K^2_L$ by
\begin{equ}
K^1_L(t,x)  := \sum_{m \in \Z^2} \big(\phi K\big)(t, x + 2\pi  m L)\;, \qquad
K^2_L(t,x)  := \sum_{m \in \Z^2} \big((1-\phi) K\big)(t, x + 2\pi m L)
\end{equ}
for $(t,x) \in \R_+ \times \T^2_L$ and such that $K_L = K^1_L + K^2_L$. On the one hand, we note that $(1-\phi) K$ is smooth at the origin as a space-time function. In particular, we have
\begin{align}
\big| \partial_x^\g  \big( (1 - \phi) K\big)(t,x) \big| \les_\g \frac{1}{\jb t^{1+ \frac {|\g|} 2}} \exp\Big(-\frac{|x|^2_{\R^2}}{10t}\Big)
\label{KerY2}
\end{align}
for $(t,x) \in \R_+ \times \R^2$ and for any multi-index $\g$. By \eqref{Poi1},~\eqref{KerY2}, and integration by parts, we then deduce the bound
\begin{align}
| \F_{\T^2_L}(K^2_L)(t,n)| = | \F_{\R^2}\big( (1- \phi)K\big)(t, n) | \les \jb n^{-100}, \quad (t,n) \in \R_+ \times \Z^2_L,
\label{KerY3}
\end{align}
which is good enough for our purposes. On the other hand, $K^1(t, \cdot)$, viewed as a function on $\T^2_L$, is supported on the disk $B(0,\frac12)$.

Now, we let $\Psi^{1}_{L,N}$ and $\Psi^2_{L,N}$ be the linear stochastic objects associated with the kernels $K^1_L$ and $K^2_L$ respectively, namely,
\begin{align}
\Psi^j_{L,N} (t,x) := \jb{\xi_L, \ind_{[0,t]} e^{-(t - \cdot)} (\widecheck{\chi_N} *_L K^j_L)(t- \cdot, x- \cdot)}
\label{KerY4}
\end{align}
for $(t,x) \in \R_+ \cdot \T^2_L$ and $j \in \{1,2\}$. Here, $\jb{\cdot, \cdot}$ denotes the $\D(\R_+ \times \T^2_L)-\D'(\R_+ \times \T^2_L)$ duality pairing and $\widecheck{\chi_N}$ denotes the (spatial) inverse Fourier transform of the function $\chi_N$ defined below \eqref{chi}. By construction, we therefore have that
\begin{align}
\Psi_{L,N} = \Psi^1_{L,N} + \Psi^2_{L,N}.
\label{KerY5}
\end{align}

We now go over basic properties of the  Bessel potentials 
on $\R^2$.
Given $\al>0$, 
 we denote by $\jb{\nabla_{\R^2}}^{-\al} = (1-\Dl_{\R^2})^{-\frac{\al}2}$
 the Bessel potential of order $\al$ on $\R^2$. It follows from \cite[Proposition 1.2.5]{Gra2}
 that for $0 < \al < 2$, 
 the convolution kernel $J_\al$ of 
 the Bessel potential $\jb{\nabla_{\R^2}}^{-\al}$ is a smooth function on $\R^2\setminus\{0\}$
that decays exponentially as $|x|\to \infty$.
 Moreover, 
we have 
 \begin{align*}
 J_\al(x) = c_{\al} |x|^{\al - 2} + o(1)
% \label{Bessel0a}
 \end{align*}
  as $x \to 0$; 
 see (4,2) in \cite{AS}. In particular, there exists a constant $c>0$ such that the following bound holds:
 \begin{align*}
 |J_\al(x)| \les 
 \begin{cases} |x|^{\al - 2},   & \text{for $0 < |x| \le 1$},  \\
 e^{-c|x|},   & \text{for $|x| >1$}.
 \end{cases}
% \label{Bessel0}
 \end{align*}
Furthermore by (the proof of) \cite[Proposition 1.2.5]{Gra2}, we have
\begin{align}
|\partial_x^{\g} J_{\al}(x)| \les_{\g} e^{-c|x|}, \quad \text{for $|x| >1$},
\label{YYY1}
\end{align}
for any multi-index $\g$.

At last, we introduce an auxiliary norm which is a slight modification of the Sobolev norms considered in Subsection~\ref{SUBSEC:N1}. Fix $0<\al<2$, a parameter $A >0$ and let $\varphi_A \in C_c^{\infty}(\R^2; [0,1])$ be a smooth bump supported on $B(0,A)$ and such that $\varphi_A \equiv 1$ on $B(0,\frac A2)$. Define the norm $\| \cdot \|_{W^{-\al,p, A}(\R^2)}$ on $W^{-\al,p}(\R^2)$ by 
\begin{align*}
\|f\|_{W^{-\al,p, A}(\R^2)} = \| (\varphi_A J_{\al}) * f \|_{L^p(\R^2)}.
%\label{Ynorm}
\end{align*}

\begin{lemma}\label{LEM:equiv}
Fix $0<\al<2$, $1 \le p \le \infty$ and $A \gg 1$. Then the norms $\|\cdot\|_{W^{-\al,p, A}(\R^2)}$ and $\|\cdot\|_{W^{-\al,p}(\R^2)}$ are equivalent.
\end{lemma}

\begin{proof}
By Plancherel's identity and Young's inequality, we bound
\begin{align}
\begin{split}
 \big\| \big( (1- \varphi_A ) J_{\al}) * f \big\|_{L^p(\R^2)} & =  \big\| \big( \jb \nb ^{100} (1- \varphi_A ) J_{\al}) * (\jb \nb ^{-100} f) \big\|_{L^p(\R^2)} \\
 & \les  \|(1- \varphi_A ) J_{\al}\|_{W^{1,1}(\R^2)} \|f\|_{W^{-100,p}(\R^2)}
\end{split}
\label{Ynorm1}
\end{align}
Furthermore, by \eqref{YYY1} and since $\varphi_A \equiv 1$ on $B(0,\frac A2)$, we have that
\begin{align}
\begin{split}
\|(1- \varphi_A ) J_{\al}\|_{W^{1,1}(\R^2)} \les e^{-c_1 A}
\end{split}
\label{Ynorm2}
\end{align}
for some constant $c_1 >0$.

By \eqref{Ynorm1} and \eqref{Ynorm2}, we then have
\begin{align*}
\| f \|_{W^{-\al,p}(\R^2)} & = \| J_\al * f \|_{L^p(\R^2)} \le \| (\varphi_A J_{\al}) * f \|_{L^p(\R^2)} +  \big\| \big( (1- \varphi_A ) J_{\al}\big) * f \big\|_{L^p(\R^2)} \\
& \les \|f \|_{W^{-\al, p, A}(\R^2)} + e^{-c_1 A} \|f\|_{W^{-100,p}(\R^2)},
\end{align*}
which shows that 
\begin{align*}
\| f \|_{W^{-\al,p}(\R^2)} \les \|f \|_{W^{-\al, p, A}(\R^2)}
\end{align*}
upon choosing $A$ large enough. The reverse inequality may be proved in a similar fashion and the claimed result follows.
\end{proof}

\begin{remark}\rm \label{RMK:equiv}
We introduce the norm $\|\cdot\|_{W^{-\al,p,A}}$ for the following practical reason: if $f$ is localized in space then so is $(\varphi_A J_\al) * f$. This will be useful in the proofs of Propositions \ref{PROP:psi} and \ref{PROP:igmc} below.
\end{remark}

We now present a proof of 
 Proposition \ref{PROP:psi}.
 
\begin{proof}[Proof of Proposition \ref{PROP:psi}] Let $L\ge 1$. Fix $T>0$, $\dl > 0$, $1 \le p < \infty$, $B \subset \R^2$ be a disk of radius at least $1$ and let $\chi_B$ be a smooth function supported on $B$ such that $\|\chi_B\|_{W^{1,\infty}(\R^2)} \les 1$. 

Our main goal is to prove the following quantitative estimates in the parameter $L$:
\begin{align}
\sup_{N \in \N}  \E\Big[ \|\Psi_{L, N}\|_{C_T \CC_x^{-\dl }(\T^2_L)}^p \Big] 
& \les_{T, p, \dl} L^{2}, \label{pY1} \\
\sup_{L \ge 1} \sup_{N \in \N}  \E\Big[ \| \chi_B \Psi_{L, N}\|_{C_T \CC_x^{-\dl}(\R^2)}^p \Big] 
& \les_{T, p, \dl} \min(|B|, L^2), \label{pY2}
\end{align}
with implicit constants independent of $L$. The convergence part of the statements (i) and (ii) and the estimates \eqref{GX0a} and \eqref{GX0b} for the limit $\Psi_{L, \infty}$ (i.e. $N = \infty$) follow from \cite{HS,ORSW} (as they are statements which do not depend on $L$ in a quantitative way).

Let $L \les 1$. The bound \eqref{pY1} follows from standard arguments, see for instance \cite{DPD, HS, Zine}. In that case, the estimate \eqref{pY2} is an immediate consequence of \eqref{pY1} and the bound
\begin{align}
\| \chi_B \Psi_{L, N}\|_{C_T \CC_x^{-\dl}(\R^2)} \les \|\chi_B\|_{W^{1,\infty}(\R^2) } \| \Psi_{L, N}\|_{C_T \CC_x^{-\dl}(\T_L^2)} \les \| \Psi_{L, N}\|_{C_T \CC_x^{-\dl}(\T_L^2)},
\label{pY3a}
\end{align}
which follows from Lemma \ref{LEM:prod}, the embedding $W^{1,\infty}(\R^2) \hookrightarrow \CC^s(\R^2)$ for $s <1$ and \eqref{B3}. Hence, we restrict ourselves to the case $L\gg 1$ in what follows. 

We focus on \eqref{pY1}. Fix $A \gg 1$ such that Lemma \ref{LEM:equiv} holds and let $L \gg A$. Let $B_{0} \subset \R^2$ be a disk of size $O(1)$ and $\mathcal B$ be a family of $ O(L^2)$ translates of $B_{0}$ forming a finitely overlapping covering of the box $\wt \T_L := [-\pi L , \pi L)^2 \subset \R^2$.
For each $B' \in \mathcal B$,  let $\phi_{B'} \in C^\infty_c(\R^2; [0, 1])$  
be a smooth function supported on $B'$ such that   $\|\phi_{B'}\|_{W^{1,\infty}} \les 1$, 
and 
$\sum_{B' \in \mathcal B}\phi_{B'}\equiv 1$ on $\wt \T_L$.

Now by \eqref{B3}, the embedding $W^{s, \infty}(\R^2) \hookrightarrow \CC^{s}(\R^2)$ for $s \in \R$ and Lemma \ref{LEM:equiv}, we have
\begin{align}
\|\Psi_{L, N}\|_{C_T \CC_x^{-\dl }(\T^2_L)} \les \|\ind_{\wt \T_L}\Psi_{L,N}\|_{C_TW^{-\dl, \infty,A}_x(\R^2)}.
\label{pY20}
\end{align}
Moreover, by the definition of the $W^{-\dl, \infty, A}$ norm, the fact that the disks in $\mathcal B$ finitely overlap and Remark \ref{RMK:equiv}, we have
\begin{align}
\begin{split}
\|\ind_{\wt \T_L}\Psi_{L,N}\|_{C_TW^{-\dl, \infty,A}_x(\R^2)} & = \Big\| (\varphi_A J_\dl) * \Big(\sum_{B' \in \mathcal B} \phi_{B'} \ind_{\wt \T_L} \Psi_{L,N}\Big)\Big\|_{C_TL_x^{\infty}(\R^2)} \\
& \les_A \big\| \sup_{B' \in \mathcal B} \, (\varphi_A J_\dl) * \big(\phi_{B'} \ind_{\wt \T_L} \Psi_{L,N}\big)\big\|_{C_TL_x^{\infty}(\R^2)} \\
& \les_A \sup_{B' \in \mathcal B} \|\ind_{\wt \T_L} \phi_{B'} \Psi_{L,N}\|_{C_TW_x^{-\dl, \infty, A}(\R^2)} \\
& \les_{A,\dl} \sup_{B' \in \mathcal B} \| \phi_{B'}\Psi_{L,N}\|_{C_T \CC_x^{-\frac \dl 2}(\T_L^2)} ,
\end{split}
\label{pY21}
\end{align}
where we used Lemma \ref{LEM:equiv} and the embedding $\CC^{s+\eps}(\R^2) \hookrightarrow W^{s, \infty}(\R^2)$ for $s\in \R$ and $\eps >0$ together with \eqref{B3} in the last inequality. Therefore, combining \eqref{pY20} and \eqref{pY21} gives
\begin{align}
\|\Psi_{L, N}\|_{C_T \CC_x^{-\dl }(\T^2_L)} \les \sup_{B' \in \mathcal B} \| \phi_{B'}\Psi_{L, N}\|_{C_T \CC_x^{-\frac \dl 2 }(\T_L^2)}.
\label{pY3}
\end{align}

Let $\ld >0$. By the union bound and since $\mathcal B$ has cardinality $\les L^2$, we have that
\begin{align}
\mathbb P \Big(\sup_{B' \in \mathcal B} \| \phi_{B'}\Psi_{L, N}\|_{C_T \CC_x^{-\frac \dl 2 }(\T_L^2)} > \ld \Big) \les L^2  \sup_{B' \in \mathcal B} \mathbb P \Big(\| \phi_{B'}\Psi_{L, N}\|_{C_T \CC_x^{-\frac \dl 2 }(\T_L^2)} > \ld \Big).
\label{pY4}
\end{align}
Besides, by \eqref{KerY4} with \eqref{KerY5}, \eqref{B3}, the embedding $W^{s,\infty}(\R^2) \hookrightarrow \CC^{s}(\R^2)$ for  $s \in \R$ and Lemma \ref{LEM:equiv} we have for any fixed $B' \in \mathcal B$
\begin{align}
\E\Big[ \|\phi_{B'}\Psi_{L, N}\|_{C_T \CC_x^{-\frac \dl 2 }(\T_L^2)}^q\Big] \les_q \E\Big[ \|\phi_{B'}\Psi^1_{L, N}\|_{C_T W^{-\frac \dl 2,\infty,A}_x(\R^2)}^q\Big]  + \E\Big[ \|\phi_{B'}\Psi^2_{L, N}\|_{C_T \CC_x^{-\frac \dl 2 }(\T_L^2)}^q\Big]
\label{pY5}
\end{align}
for any finite $q \ge 1$. Since $\Psi^2_{L,N}$ is constructed from the smooth kernel $K^2_L$, a standard analysis using \eqref{KerY3} and the fact the disk $B'$ has support of size $O(1)$ (see \cite{DPD, HS, MW2}) shows that
\begin{align}
\E\Big[ \|\phi_{B'}\Psi^2_{L, N}\|_{C_T \CC_x^{10}(\T_L^2)}^q\Big] \les_{T,q} 1,
\label{pY6}
\end{align}
with an implicit constant independent of $L$ and $N$.

As for $\Psi^1_{L,N}$, we note that for each fixed $x \in \R^2$, $(\varphi_A J_{\dl}) * (\phi_{B'} \Psi^1_{L,N})(x)$ only involves spatial values of the noise $\xi_L$ on a small box contained in the dilate $cA \cdot B'$, for some $c>0$, of the disk $B'$. More precisely $(\varphi_A J_{\dl}) * (\phi_{B'} \Psi^1_{L,N}) = (\varphi_A J_{\dl}) * (\phi_{B'} \wt\Psi^2_{L,N} )$, with $\wt \Psi^1_{L,N}$ given by
\begin{align}
\wt \Psi^1_{L,N} (t,x) := \jb{\xi_L, \ind_{\wt \T_A} \ind_{[0,t]} e^{-(t - \cdot)} (\widecheck{\chi_N} *_L K^1_L)(t- \cdot, x- \cdot)}
\label{pY6b}
\end{align}
for $(t,x) \in \R_+ \times \T^2_L$, where $\wt \T_{A}$ is a box of side-length $L_0 \sim A$ and which shares the same center as $B'$. By \eqref{noise}, \eqref{KerY4} and \eqref{pY6b}, we observe that $\wt \Psi^1_{L,N}$ and $\Psi^1_{L_0,N}$ have the same law and thus
\begin{align}
\begin{split}
\E\Big[ \|\phi_{B'} \Psi^1_{L, N}\|_{C_T W^{-\frac \dl 2,\infty,A}_x(\R^2)}^q\Big] & = \E\Big[ \|\phi_{B'} \wt \Psi^1_{L, N}\|_{C_T W^{-\frac \dl 2,\infty,A}_x(\R^2)}^q\Big]\\
& = \E\Big[ \|\phi_{B'} \Psi^1_{L_0, N}\|_{C_T W^{-\frac \dl 2,\infty,A}_x(\R^2)}^q\Big] \\
& \sim \E\Big[ \|\phi_{B'} \Psi^1_{L_0, N}\|_{C_T W^{-\frac \dl 2,\infty}_x(\R^2)}^q\Big] \\
&  \les \E\Big[ \|\phi_{B'} \Psi^1_{L_0, N}\|_{C_T \CC^{-\frac \dl 4,\infty}_x(\R^2)}^q\Big]\\
& \les \E\Big[ \|\phi_{B'} \Psi^1_{L_0, N}\|_{C_T \CC^{-\frac \dl 4,\infty}_x(\T_{L_0}^2)}^q\Big],
\end{split}
\label{pY7}
\end{align}
where we used Lemma \ref{LEM:equiv}, the embedding $\CC^{s+\eps}(\R^2) \hookrightarrow W^{s,\infty}(\R^2)$, for $s \in \R$ and $\eps >0$ and \eqref{B3} in the last three steps, respectively. By \eqref{pY2} for $L_0 \les_A 1$ together with \eqref{pY6} (for $L_0$) and \eqref{KerY5} and since $L_0 \sim A$, we then have
\begin{align}
\E\Big[ \|\phi_{B'} \Psi^1_{L_0, N}\|_{C_T \CC^{-\frac \dl 4,\infty}_x(\T_{L_0}^2)}^q\Big] \les_{T,q,\dl, A} 1.
\label{pY8}
\end{align}

Thus, by combining \eqref{pY4}, \eqref{pY5}, \eqref{pY6}, \eqref{pY7} and \eqref{pY8} together with Chebyshev's inequality, we get that
\begin{align}
\begin{split}
\E \Big[\sup_{B' \in \mathcal B} \| \phi_{B'}\Psi_{L, N}\|_{C_T \CC_x^{-\frac \dl 2}(\T_L^2)}^p\Big] & = \int_{0}^{\infty} \mathbb P \Big(\sup_{B' \in \mathcal B} \| \phi_{B'}\Psi_{L, N}\|_{C_T \CC_x^{-\frac \dl 2 }(\T_L^2)} > \ld^{\frac1p}\Big) d \ld \\
& \les_{T,q, \dl} L^2 \int_0^{\infty} \ld^{-\frac q p} d \ld \\
&  \les_{T,q, \dl} L^2,
\label{pY9}
\end{split}
\end{align}
upon choosing $q = 10p$, say. The bound \eqref{pY1} for $L \gg 1$ thus follows from \eqref{pY3} and \eqref{pY9}.

The estimate \eqref{pY2} for $L\gg1$ follows from a similar covering argument and \eqref{pY1} together with \eqref{pY3a}. We omit details.
\end{proof}

We conclude this section by presenting 
a proof of
Proposition \ref{PROP:igmc}.

\begin{proof}[Proof of Proposition~\ref{PROP:igmc}]Let $L\ge 1$ and $0 < \be^2 < 4\pi$. Fix $T>0$, $\al > \frac{\be^2}{4\pi}$, and $1 \le p < \infty$.

As in the proof of Proposition \ref{PROP:psi}, we may reduce the bounds \eqref{YY1} and \eqref{YY2} to the case $L \les 1$ via a covering argument. Let us sketch the modifications to make to the proof of \eqref{pY1} to obtain \eqref{YY1}. In particular, as in \eqref{pY5}-\eqref{pY9}, it suffices to prove
\begin{align}
\E\Big[ \|\phi_{B'}\Ta_{L, N}\|_{C_T \CC_x^{-\al}(\T_L^2)}^q\Big] \les_{T,q,\al} 1,
\label{YY5}
\end{align}
for all finite $q \ge 1$ and all disk $B'$ of radius at least 1.

By \eqref{igmc}, the product estimate (Lemma \ref{LEM:prod}), \eqref{KerY5}, Cauchy-Schwarz's inequality and \eqref{KerY3}, we have
\begin{align}
\begin{split}
\E\Big[ \|\phi_{B'}\Ta_{L, N}\|_{C_T \CC_x^{-\al}(\T_L^2)}^q\Big] & \les \E \Big[ \big\|\phi_{B'}(\g_{L,N} e^{i \be \Psi^1_{L,N}})\big\|_{C_T \CC_x^{-\al}(\T_L^2)}^{2q}\Big]^{\frac12}  \E\Big[ \|e^{i \be \Psi^2_{L,N}}\|_{C_T \CC_x^{10}(\T_L^2)}^{2q}\Big]^{\frac12} \\
& \les_{T,q} \E \Big[ \big\|\phi_{B'}(\g_{L,N} e^{i \be \Psi^1_{L,N}})\big\|_{C_T \CC_x^{-\al}(\T_L^2)}^{2q}\Big]^{\frac12}.
\end{split}
\label{YY6}
\end{align}
Now, as in the proof of Lemma \ref{PROP:psi}, we have that $\phi_{B'}(\g_{L,N} e^{i \be \Psi^1_{L,N}})$ and $\phi_{B'}(\g_{L,N} e^{i \be \Psi^1_{L_0,N}})$, for some $L_0 \sim 1$, share the same law. By \eqref{B3}, \eqref{CN}, \eqref{sN} and \eqref{igmc}, this leads to
\begin{align}
\begin{split}
& \E \Big[ \big\|\phi_{B'}(\g_{L,N} e^{i \be \Psi^1_{L,N}})\big\|_{C_T \CC_x^{-\al}(\T_L^2)}^{2q}\Big]^{\frac12} \\
& \qquad \quad = \E \Big[ \big\|\phi_{B'}(\g_{L,N} e^{i \be \Psi^1_{L_0,N}})\big\|_{C_T \CC_x^{-\al}(\T_{L_0}^2)}^{2q}\Big]^{\frac12} \\
& \qquad \quad = e^{\frac{\be^2}{2}(\s_{L,N} - \s_{L_0,N})} \E \Big[ \big\|\phi_{B'}(\g_{L_0,N} e^{i \be \Psi^1_{L_0,N}})\big\|_{C_T \CC_x^{-\al}(\T_{L_0}^2)}^{2q}\Big]^{\frac12}\\
& \qquad \quad \les \E \Big[ \big\|(\phi_{B'}\Ta_{L_0,N}) e^{-i \be \Psi^2_{L_0,N}}\big\|_{C_T \CC_x^{-\al}(\T_{L_0}^2)}^{2q}\Big]^{\frac12} \\
& \qquad \quad \les \E \Big[\|\phi_{B'}\Ta_{L_0,N}\|_{C_T \CC_x^{-\al}(\T_{L_0}^2)}^{4q}\Big]^{\frac14} \E \Big[\|e^{-i\Psi^2_{L_0,N}} \|_{C_T \CC_x^{10}(\T_{L_0}^2)}^{4q}\Big]^{\frac14} \les_{T,q,\al} 1,\\
\end{split}
\label{YY7}
\end{align}
where we used \eqref{KerY5}, Lemma \ref{LEM:prod}, Cauchy-Schwarz's inequality, \eqref{KerY3} and \eqref{YY1} for $L_0 \les 1$ in the last two lines.

Hence, \eqref{YY5} follows from combining \eqref{YY6} and \eqref{YY7}. The rest of the proof follows from similar arguments and we omit details.
\end{proof}

\begin{remark}
In \cite{HS}, 
the second author with Shen constructed the imaginary Gaussian multiplicative chaos  $\Ta
= \lim_{N \to \infty} \Ta_{L = 1, N}$ in the 
subcritical full range $0 < \be^2 < 8\pi$ by 
working with negative temporal regularity.
See also \cite{OZ}
for an analogous construction
the imaginary Gaussian multiplicative chaos 
 in 
 the hyperbolic case, 
 where   the range of $\be^2$ is smaller ($0 < \be^2 < 6\pi$), 
 exhibiting a sharp contrast to the parabolic case.
 Since we restrict our attention to $0 < \be^2 < 4\pi$ in the current paper, 
we do not need such an argument.
\end{remark}
\section{Proof of Theorem~\ref{THM:main}}\label{SEC:proof}

\subsection{Key deterministic a priori bound}
In this subsection, we fix $L\ge 1$ and
we establish a deterministic global-in-time a priori bound
for a nonlinear heat equation on $\T^2_L$, 
which plays a crucial role in the proof of Theorem~\ref{THM:main}
presented in the next subsection.

Given  $s,s_0>0$
and $T > 0$, 
we define the space $X^{-s_0,s}_T(\T_L^2)$ as $ C((0,T]; \CC^{s}(\T_L^2))$ endowed with the norm:
\begin{align}
\|v\|_{X^{-s_0,s}_T(\T^2_L)} = \sup_{0<t\le T} e^t \min(1,t)^{\frac{s+s_0}{2}} \|v(t)\|_{\CC_x^s(\T^2_L)}.
\label{XX0}
\end{align}

Our goal is to study the following nonlinear heat equation:
\begin{align}
\dt v + (1- \Dl) v = \sum_{\kk \in \{+, -\}} f^\kk(v) \Ta^\kk , \quad (t,x) \in \R_+ \times \T^2_L,
\label{heat1}
\end{align}
where $f^\pm(v) = e^{\pm i \be v}$
and $\Ta^\kk$, $\kk \in \{+, -\}$, 
are given smooth functions.

\begin{proposition}\label{PROP:key}
Let $L \ge 1$, $0 < \be^2 < 4\pi$, and  $\dl >0$ sufficiently small such that
\begin{align}
s: = \frac{\be^2}{4\pi} + 3\dl<1.
\label{cond}
\end{align}
Then, 
there exist $K_1, A>0$ 
and a non-decreasing function $K_2: \R_+ \to \R_+$,
independent of $L \ge 1$,  such that
\begin{align}
\|v\|_{X^{-\dl,s}_T(\T^2_L)} \le   K_1 \|v(0)\|_{\CC^{-\dl}(\T^2_L)} + K_2(T)  
\sum_{\kk \in \{+, -\}}
\|\Ta^\kk\|_{C_T \CC^{2\dl-s}_x(\T^2_L)}^{A}
\label{X0}
\end{align}
 for any $T \ge 1$
 and any solution $v$ to \eqref{heat1} on $[0, T]$.
\end{proposition}

\begin{proof}%[Proof of Proposition~\ref{PROP:key}]
Fix $T\ge 1$ and $0 < t \le T$. 
In the following, we only consider the case  $0 < t \le 1$, 
since the case $1 < t \le T$ follows in a similar manner.

By applying Lemma~\ref{LEM:Schauder} and Minkowski's inequality
to the Duhamel formulation of \eqref{heat1}:
\begin{align}
v(t) = e^{t (\Dl - 1)}v(0)  + \sum_{\kk \in \{+, -\}}\I ( f^\kk(v) \Ta^\kk )(t), 
\label{X0a}
\end{align}
where $\I$ is as in \eqref{duha}, 
we have
\begin{align}
\begin{split}
  e^t t^{\frac{s+\dl}{2}}\| v(t) \|_{\CC^s(\T^2_L)}
%& \quad 
&    \les \|v(0)\|_{\CC^{-\dl}(\T^2_L)} \\
& \quad  + e^t t^{\frac{s+\dl}{2}}
\sum_{\kk \in \{+, -\}}
\int_0^t \big\| e^{(t-t')(\Dl - 1)} \big(f^\kk(v) \Ta^\kk\big)(t') \big\|_{\CC_x^s(\T^2_L)} dt'.
\end{split}
\label{X1}
\end{align}
By Lemmas~\ref{LEM:Schauder},~\ref{LEM:prod}, and~\ref{LEM:elem}
with \eqref{cond}, we have
\begin{align}
\begin{split}
& e^t t^{\frac{s+\dl}{2}} \int_0^t \big\| e^{(t-t')(\Dl - 1)} \big(f^\kk(v) \Ta^\kk\big) (t')\big\|_{\CC_x^s(\T^2_L)} dt'\\
&  \quad \les t^{\frac{s+\dl}{2}} \int_0^t (t-t')^{-(s-\dl)} (t')^{-\frac{s+\dl}{2}} \| e^{t'} (t')^{\frac{s+\dl}{2}} f^\kk (v)(t') \Ta^\kk(t') \|_{\CC_x^{2\dl-s}(\T^2_L)} dt' \\
&  \quad \les  t^{\frac{s+\dl}{2}} t^{1 + \frac{\dl-3s}{2}} \cdot \sup_{0< t \le T} \| e^{t} t^{\frac{s+\dl}{2}} f^\kk (v)(t) \Ta^\kk(t) \|_{\CC_x^{2\dl-s}(\T^2_L)} \\
&  \quad \les T^{1+ \dl-s}  \| \Ta^\kk \|_{C_T\CC_x^{2\dl-s}(\T^2_L)} \cdot \sup_{0< t \le T} e^{t} t^{\frac{s+\dl}{2}} \|  f^\kk (v)(t) \|_{\CC_x^{s-\dl}(\T^2_L)}
\end{split}
\label{X2}
\end{align}
for $\kk \in \{+, -\}$.
From  interpolation and the boundedness of $f^\kk$, we have 
\begin{equation}
\|f^\kk(v)\|_{\CC^{s-\dl}(\T^2_L)} 
 \les \|f^\kk(v)\|_{\CC^s(\T^2_L)}^{1-\ta}
\|f^\kk(v)\|_{L^\infty(\T^2_L)}^{\ta}
\les \|v\|^{1-\ta}_{\CC^{s}(\T^2_L)},
\label{X3}
\end{equation}
provided that  $\ta >0$ is sufficiently small such that 
$(1-\ta)s  \ge s-\delta$.
Hence, putting \eqref{X1}, \eqref{X2} and \eqref{X3} together, 
we obtain
\begin{align}
\|v\|_{X^{-\dl, s}_T(\T_L^2)} \les \|v(0)\|_{\CC^{-\dl}(\T^2_L)} + 
T^{\g} e^{\ta T} 
\sum_{\kk \in \{+, -\}}
  \| \Ta^\kk \|_{C_T \CC_x^{2\dl-s}(\T^2_L)} \|v\|_{X^{-\dl,s}_T(\T_L^2)}^{1-\ta}
\label{X4}
\end{align}
for some  $\g >0$.
By applying  Young's inequality, 
it follows from  \eqref{X4} that 
\begin{align*}
\|v\|_{X^{-\dl, s}_T(\T^2_L)} \le C_1 \|v(0)\|_{\CC^{-\dl}(\T^2_L)} + 
C_2(T)  \sum_{\kk \in \{+, -\}} \| \Ta^\kk \|_{C_T \CC_x^{2\dl-s}(\T^2_L)}^\frac1 \ta + \frac{1}{2} 
\|v\|_{X^{-\dl,s}_T(\T^2_L)}
\end{align*}
for some $C_1, C_2(T) > 0$, 
from which we obtain \eqref{X0}.
%This concludes  the proof
%of Proposition~\ref{PROP:key}.
\end{proof}

\begin{remark}
Let $\al = \frac{\be^2}{4\pi} + \dl$, where $\dl > 0$
is as in Proposition~\ref{PROP:key}.
Then,  the proof of 
 Proposition~\ref{PROP:key}
requires 

 \begin{enumerate}[leftmargin=2em]
 \item\label{item:first}
  $s- 2 < - \al$ at the first inequality in \eqref{X2}, 
  where we used  almost two  degrees
  of smoothing from the heat Duhamel integral operator, 

\item
$\al < s$
in applying the product estimate (Lemma~\ref{LEM:prod})
at the last step of \eqref{X2}.

  \end{enumerate}
Hence, we have $\al< s < 2- \al$,
in particular $\al< 1$, 
from which we obtain the restriction $\be^2 < 4\pi$
in Proposition~\ref{PROP:key}
and hence in Theorem~\ref{THM:main}.

In Appendix \ref{SEC:wave}, 
we establish
an analogous global-in-time a priori bound
for the damped nonlinear wave equation \eqref{WW9}.
Due to one degree of smoothing under the damped
wave Duhamel integral operator, 
the restriction $s - 2 <  - \al $ in \ref{item:first} above
is replaced by $s - 1 \le - \al$.
%see \eqref{ZZ6}.
Together with $\al < s$, 
we obtain $\al < s \le 1 - \al$, in particular $\al < \frac 12$.
With $\al = \frac{\be^2}{4\pi} + \dl$, 
this gives a restriction $\be^2 < 2\pi$
in Theorem~\ref{THM:GWP}
on pathwise global well-posedness
of the hyperbolic sine-Gordon model.
\end{remark}

\subsection{Proof of Theorem~\ref{THM:main}} 
 Fix $L\ge 1$ and  $0 < \be^2 < 4\pi$, 
and  small $\dl > 0$, 
 and let $s$ be as in~\eqref{cond}.
 Let $T \ge 1$ to be chosen later and $1 \le p < \infty$.  
 Given $N \in \N$, let $v_{L, N}$ be the solution to~\eqref{v1}, % with initial data $u_{N}(0) - \Psi(0)$, 
 where $ \Law(u_{L, N}(0)) = \rho_{L, N}^\be$. 
Then, we have 
\begin{align*}
v_{L, N} = e^{t (\Dl - 1)} \big( u_{L, N}(0) - \Psi_L(0)\big)  
- \frac{1}{2i} 
\sum_{\kk \in \{+, -\}} \kk \Pii_N \I (  f^\kk ( v_{L, N}) \Ta^\kk_{L, N}), 
\end{align*} 
where $f^\pm(v) = e^{\pm i \be v}$, 
$\Ta_{L, N}^+ = \Ta_{L, N} =  \g_{L, N}  e^{ i \be \Psi_{L, N}}$
as in \eqref{igmc}, 
and 
$\Ta_{L, N}^- =  (\Ta_{L, N})^{-1} = \g_{L, N}  e^{- i \be \Psi_{L, N}}$.

Let  $s, \dl >0$ be as in Proposition~\ref{PROP:key}.
Then, by Proposition~\ref{PROP:key} 
 with \eqref{XX0}
 and Propositions~\ref{PROP:psi} and~\ref{PROP:igmc}, we have
 \begin{align}
\begin{split}
& \E \big[ \| v_{L, N}(T) \|_{\CC^s(\T^2_L)}^p \big] \\
& \quad \le C_p
 \bigg(K_1^p e^{-pT} \Big( \E \big[ \|u_{L, N}(0)\|_{\CC^{-\dl}(\T^2_L)}^p \big] 
+ \E \big[ \|\Psi_L(0)\|_{ \CC^{-\dl}(\T^2_L)}^p \big]\Big) \\
&  \quad\quad    + e^{-pT} K_2^p(T) \, \E \big[ \| \Ta_{L, N} \|_{C_T \CC^{-\frac{\be^2}{4\pi}-\dl}_x(\T^2_L)}^{A p} \big] 
\bigg)\\
& \quad \le \wt K_1 e^{-pT}\,  \E \big[ \|u_{L, N}(0)\|_{\CC^{-\dl}(\T^2_L)}^p \big] + \wt K_2(T,L)
\end{split}
\label{Y1}
\end{align}
for some  $\wt K_1, \wt K_2(T,L) >0$, 
uniformly in $N \in \N$.
Here, we used the fact that $\Law (\Ta_{L, N}^{-} )= \Law (\Ta^+_{L, N} )$
since $\Law (-\Psi_{L, N})= \Law (\Psi_{L, N} )$.

On the other hand, by the invariance of the truncated
sine-Gordon measure 
 $\rho_{L, N}^\be$ (Lemma~\ref{LEM:inv1}), \eqref{expa}, Proposition~\ref{PROP:psi}, 
  and \eqref{Y1}, we have
\begin{equs}
\E &\big[ \| u_{L, N}(0) \|_{\CC^{-\dl}(\T^2_L)}^p \big] 
 = \E \big[ \| u_{L, N}(T) \|_{\CC^{-\dl}(\T^2_L)}^p \big] \\
& \quad \les   \E \big[ \| \Psi_L(T) \|_{\CC^{-\dl}(\T^2_L)}^p \big] +  \E\label{Y2}
 \big[ \| v_{L, N}(T) \|_{\CC^{-\dl}(\T^2_L)}^p \big] \\
& \quad \le \wt K_1 e^{-pT} \, \E \big[ \|u_{L, N}(0)\|_{\CC^{-\dl}(\T^2_L)}^p \big] 
+ \wt K_3(T,L)
\end{equs}

for some $\wt K_3(T,L) > 0$,
uniformly in $N \in \N$.
By choosing sufficiently large $T \gg 1$
such  that $\wt K_1 e^{-pT} < \frac12$, we deduce from \eqref{Y2} that 
\begin{align*}
\sup_{N \in \N} \E \big[ \| u_{L, N}(0) \|_{\CC^{-\dl}(\T^2_L)}^p \big]  < \infty.
\end{align*}
Since the choice of small $\delta>0$ was arbitrary, it follows that the sequence $\{\rho_{L, N}^\be\}_{N \in \N}$
is tight on $\CC^{-\dl}(\T_L^2)$, 
and hence  Theorem~\ref{THM:main} follows from Prokhorov's theorem.

\section{Sine-Gordon model in the full space}
\label{SEC:R2}

In this section, 
by slightly modifying the argument presented in Section 
\ref{SEC:proof}, we present a proof of 
Theorem \ref{THM:main2}
on the construction of 
 the sine-Gordon model in $\R^2$ in the range $0 < \be^2 < 4\pi$.

Given $M \gg1$, 
let $\chi_0, \chi_1 \in C^{\infty}_c(\R^2; [0, 1])$ with 
$\supp (\chi^M_0) \subset \big\{|x|\le \tfrac{4}{3}M\big\}$ and $
\supp (\chi^M_1) \subset \big\{\tfrac{3}{4}M \le |x|\le \tfrac{4}{3}M^2\big\}$ 
such that, setting  $\chi^M_\l(x) = \chi_1(M^{1-\l} x)$ for $\l \ge 2$, 
$\{\chi^M_\l\}_{\l \in \Z_{\ge 0}}$ is an 
 $M$-adic partition of unity on the plane $\R^2$, namely
\begin{align}
\sum_{\l = 0}^\infty \chi^M_\l(x) =  1
\label{A0}
\end{align}
for any $x \in \R^2$.
Note that for $|k-\l| \ge 2$, 
we have 
\begin{align}
\chi^M_k \chi^M_\l = 0.
\label{A1}
\end{align}
Moreover,  we may assume that 
\begin{align}
\|\nb \chi^M_\l \|_{L^\infty}
+ 
\|\Dl \chi^M_\l \|_{L^\infty}    \les M^{-1} \ll 1, 
\label{A2}
\end{align}
uniformly in $\l \in \Z_{\ge 0}$.

Given $s \in \R$ and $\ld \gg 1$, 
we define a weighted H\"older-Besov space $\C^s_{\ld, M}(\R^2)$
as the completion of $C^\infty_c(\R^2)$
with respect to the norm
\begin{align}
\|v\|_{\C^s_{\ld, M}(\R^2)} = \sum_{\l =  0}^\infty e^{-\ld \l} \|\chi^M_\l v\|_{\CC_x^s(\R^2)}.
\label{A3}
\end{align}
See also \cite{MW2, OTWZ}
for different versions of weighted Besov and Sobolev spaces.

By a slight modification of the proof of \cite[Lemma 2.2]{OTWZ}, 
we have the following lemma on compact embedding
of the weighted H\"older-Besov space $\CC^{s}_{\ld, M} (\R^2)$.
Since the required modification is straightforward, we omit its proof.

\begin{lemma}\label{LEM:cpt}
Let $M \gg 1$.
Then,  given any $s > s'$ and $\ld < \ld'$, 
the embedding
\[
\CC^{s}_{\ld, M} (\R^2) \hookrightarrow 
\CC^{s'}_{\ld', M} (\R^2)
\]
is compact.

\end{lemma}

As in \eqref{XX0}, 
given $s, s_0> 0$ and $T > 0$, 
we define the space $X^{-s_0,s}_T(\R^2)$ as $ C((0,T]; \CC^{s}(\R^2))$ endowed with the norm:
\begin{align}
\|v\|_{X^{-s_0,s}_T(\R^2)} = \sup_{0<t\le T} e^t \min(1,t)^{\frac{s+s_0}{2}} \|v(t)\|_{\CC_x^s(\R^2)}.
\label{A4}
\end{align}
Then, we define $Y^{-s_0,s}_T(\R^2)
= Y^{-s_0,s}_{T, \ld, M}(\R^2)$ 
by the norm: 
\begin{align}
\|v\|_{Y^{-s_0,s}_T(\R^2)}  = \sum_{\l = 0}^\infty e^{-\ld \l}  \|\chi^M_\l v\|_{X^{-s_0, s}_T(\R^2)}.
\label{A5}
\end{align}
Given $A > 0$, we also define
$Z^{s, A}_T(\R^2)= Z^{s, A}_{T, \ld, M}(\R^2)$ by the norm:
\begin{align}
\|\Ta\|_{Z^{s, A}_{T}(\R^2)} & = \sum_{\l =  0}^\infty e^{-\ld \l} \|\chi^M_\l\Ta\|^A_{C_T\CC^s_x(\R^2)}.
\label{A6}
\end{align}

We now consider the following nonlinear heat equation:
\begin{align}
\dt v + (1- \Dl) v = \sum_{\kk \in \{+, -\}} f^\kk(v) \Ta^\kk , \quad (t,x) \in \R_+ \times \R^2,
\label{heat2}
\end{align}
where $f^\pm(v) = e^{\pm i \be v}$
and $\Ta^\kk$, $\kk \in \{+, -\}$, 
are given smooth functions.
As a slight modification of the proof of Proposition \ref{PROP:key}, 
we establish an a priori bound of solutions to \eqref{heat2}.

\begin{proposition}\label{PROP:key2}
Let $0 < \be^2 < 4\pi$ and  $\dl >0$ sufficiently small such that
%$s: = \frac{\be^2}{4\pi} + 3\dl<1$
%\eqref{cond} holds.
\begin{align}
s: = \frac{\be^2}{4\pi} + 3\dl<1.
\label{cond2}
\end{align}

Then, 
there exist  $K_1, A>0$ and a function  $K_2: \R_+^2 \to \R_+$
which is non-decreasing in each argument
such that
\begin{align}
\|v\|_{Y^{-\dl,s}_T(\R^2)} 
& \le K_1  
\|v(0)\|_{\CC^{-\dl}_{\ld, M}(\R^2)}
 + K_2(T,M)  
\sum_{\kk \in \{+, -\}}
\|\Ta^\kk\|_{Z^{2\dl-s, A}_T(\R^2)},
\label{Z0}
\end{align}
for any  $T\ge 1$ and $\ld,M>0$, satisfying  
\begin{align}
T e^{\ld} M^{-1} & \ll 1.
\label{para2}
\end{align}

\end{proposition}

\begin{proof}
Fix $T\ge 1$ and $0 < t \le T$. 
As in the proof of Proposition \ref{PROP:key}, 
we only consider the case  $0 < t \le 1$, 
since the case $1 < t \le T$ follows in a similar manner.

For simplicity, 
we write $v_\l$ for $\chi^M_\l v$. 
Then, from \eqref{A0} and \eqref{A1}, we have 
\begin{align}
\chi^M_\l f^\pm(v) = \chi^M_\l e^{\pm i \be v} 
= \chi^M _\l \prod_{k = 0}^\infty e^{\pm i \be v_k}  = 
\chi^M_\l
\prod_{ |k - \l| \le 1}
f^\pm(v_k)
\label{Z1}
\end{align}
for each $\l \in \Z_{\ge 0}$.
Thus, by 
letting $\L = \dt + 1 -\Dl $, it follows from \eqref{heat2} and \eqref{Z1}
with \eqref{A0} and \eqref{A1}
 that 
\begin{align}
\begin{split}
\L (v_\l) & = \L(\chi^M_\l v) \\
& = - \Dl \chi_\l^M v -  \nb \chi^M_\l \cdot \nb v + \sum_{\kk \in \{+, -\}} 
\bigg(\prod_{ |k - \l| \le 1} f^\kk(v_k) \bigg)\chi^M_\l   \Ta^\kk\\
& = - \sum_{ |k - \l| \le 1}\big(\Dl \chi_\l^M v_k +   \nb \chi^M_\l \cdot \nb v_k\big)
+ \sum_{\kk \in \{+, -\}} 
\bigg(\prod_{ |k - \l| \le 1} f^\kk(v_k) \bigg)\chi^M_\l   \Ta^\kk
\end{split}
\label{Z2}
\end{align}
for each $\l \in \Z_{\ge 0}$.

Fix $\l \in \Z_{\ge 0}$.
Then, 
by arguing as in \eqref{X2}-\eqref{X4} 
with Young's inequality, we have 
\begin{align}
\begin{split}
& e^t t^{\frac{s+\dl}{2}} \int_0^t 
\bigg\| e^{(t-t')(\Dl - 1)} \bigg(\prod_{ |k - \l| \le 1} f^\kk(v_k)(t')\bigg) 
\chi^M_\l \Ta^\kk (t')\bigg\|_{\CC_x^s(\R^2)} dt'\\
& \qquad \quad \le C(T,M)  \sum_{\kk \in \{+, -\}} \|\chi^M_\l \Ta^\kk \|_{C_T \CC_x^{2\dl-s}(\R^2)}^\frac1 \ta + M^{-1} \sum_{ |k - \l| \le 1} \|v_k\|_{X^{-\dl,s}_T(\R^2)}
\end{split}
\label{Z3}
\end{align}
for some small $0 < \ta \ll 1$. 
Proceeding as in  \eqref{X2}
with \eqref{cond2}, \eqref{A4}, 
and  \eqref{A2}, we have
\begin{align}
\begin{split}
e^t t^{\frac{s+\dl}{2}} \int_0^t \big\| e^{(t-t')(\Dl - 1)} \big( \Dl \chi_\l^M v_k(t') \big) \big\|_{\CC_x^s(\R^2)} dt' 
& \les T \| \Dl \chi^M_\l \|_{L^\infty} \|v_k\|_{X^{-\dl,s}_T(\R^2)} \\
& \les T M^{-1}\|v_k\|_{X^{-\dl,s}_T(\R^2)}, 
\end{split}
\label{Z4}
\end{align}
uniformly in $\l, k \in \Z_{\ge 0}$ with $|k - \l |\le 1$.
Similarly, by Lemmas \ref{LEM:Schauder}, \ref{LEM:prod}, and \ref{LEM:elem}
with  \eqref{A2} and \eqref{cond2}, we have
\begin{align}
\begin{split}
 e^t & t^{\frac{s+\dl}{2}} \int_0^t \big\| e^{(t-t')(\Dl - 1)} 
 \big( \nb \chi_\l^M \cdot \nb v_k (t') \big) \big\|_{\CC_x^s(\R^2)} 
dt' \\
& \les  t^{\frac{s+\dl}{2}} \int_0^t (t-t')^{-\frac12} (t')^{-\frac{s+\dl}{2}}\| e^{t'} (t')^{\frac{s+\dl}{2}} \nb \chi_\l^M \cdot \nb v_k (t') \|_{\CC_x^{s-1}} dt' \\
& \le T^{\frac 12 } \| \nb \chi_\l^M \|_{\CC^1} \|v_k\|_{X^{-\dl,s}_T(\R^2)} \\
& \les T^{\frac 12 } M^{-1}\|v_k\|_{X^{-\dl,s}_T(\R^2)}, 
\end{split}
\label{Z5}
\end{align}
uniformly in $\l, k \in \Z_{\ge 0}$ with $|k - \l |\le 1$.

Hence, by applying
Lemma \ref{LEM:Schauder}, 
  \eqref{Z2},  \eqref{Z3}, \eqref{Z4} and \eqref{Z5} 
  to the Duhamel formulation~\eqref{X0a}, we obtain
\begin{align}
\begin{split}
\|v_\l\|_{X^{-\dl,s}_T(\R^2)} & \les \|v_\l(0)\|_{\CC^{-\dl}(\R^2)} + T  M^{-1} 
\sum_{ |k - \l| \le 1} \|v_k\|_{X^{-\dl,s}_T(\R^2)} \\
& \quad  + C(T,M)  \sum_{\kk \in \{+, -\}} \|\chi^M_\l \Ta^\kk \|_{C_T \CC_x^{2\dl-s}(\R^2)}^\frac1 \ta,
\end{split}
\label{Z6}
\end{align}
uniformly in $\l \in \Z_{\ge 0}$.
Then, from \eqref{Z6} with \eqref{A5}, \eqref{A3} and \eqref{A6}, we obtain
\begin{align*}
\|v\|_{Y^{-\dl,s}_T(\R^2)} 
& \les 
\|v(0)\|_{\CC^{-\dl}_{\ld, M}(\R^2)} + T  M^{-1} e^{\ld}\|v\|_{Y^{-\dl,s}_T(\R^2)} \\
& \quad  + C(T,M)  
\sum_{\kk \in \{+, -\}}
\|\Ta^\kk\|_{Z^{2\dl-s, \frac{1}{\ta}}_T(\R^2)},
\end{align*}
which 
yields the desired bound
 \eqref{Z0} under the assumption \eqref{para2}.
\end{proof}

We now present a proof of Theorem \ref{THM:main2}.

\begin{proof}[Proof of Theorem \ref{THM:main2}]

Given $L \ge 1$ and $N \in \N$, 
let $\rho^{\be}_{L, N}$ be the truncated sine-Gordon measure in \eqref{mes3},
 viewed 
as a probability measure on  $\mathcal D'(\R^2)$, 
but restricted to 
$2\pi L$-periodic distributions  on $\R^2$.
We fix $T \gg 1$, $M = M(T) \gg1$, $\ld=\ld(T)\gg1 $, 
satisfying~\eqref{para2}.
Given   $A > 0$  as in Proposition \ref{PROP:key2}, 
we choose $p \ge 1$ sufficiently large such that 
\begin{align}
\begin{split}
e^{-\ld} M^{\frac 3 p} & \ll 1.
\end{split}
\label{para}
\end{align}

 Fix 
 $0 < \be^2 < 4\pi$ and small $\dl > 0$
 and let $s$ be as in \eqref{cond2}.
 Given $L \ge 1$, 
 let $\Psi_L$ be as in~\eqref{Psi0}, 
 where we view $\xi_L$ as the $2\pi L$-periodized (in space) white noise
 on $\R \times \R^2$; see \cite[Section 5]{MW2}.
Then, 
let 
  $\Ta_{L, N}$ 
be as in  \eqref{igmc} viewed as a function on $\R_+\times \R^2$, $2\pi L$-periodic in the spatial variable.
Then, from Proposition \ref{PROP:igmc}, we have 
\begin{align}
\E \big[ \| \chi^M_\l \Ta_{L, N}^\kk\|_{C_T \CC_x^{2\dl-s}(\R^2)}^{pA} \big]
 \les_{T,p, A,s,\dl} M^{2(\l+1)}, 
\label{Z8}
\end{align}
uniformly in $L \ge1$, $N \in \N$, $\l \in \Z_{\ge 0}$, 
and 
 $\kk \in \{-1,+1\}$.
Then, from the definition 
 \eqref{A6}, Minkowski's inequality,  \eqref{Z8}, and \eqref{para}, we have
\begin{align}
\begin{split}
\Big\|\|\Ta_{L, N}^\kk\|_{Z^{2\dl-s, A}_T(\R^2)}\Big\|_{L^{p}(\PP)} 
& \le \sum_{\l = 0}^\infty e^{-\ld \l} 
\Big\|\|\chi^M_\l \Ta_{L, N}^\kk\|_{C_T \CC_x^{2\dl-s}(\R^2)}^A\Big\|_{L^{p}(\PP)} \\
& \les \sum_{\l = 0}^\infty e^{-\ld \l} M^{\frac 2 p(\l+1)} \les 1,
\end{split}
\label{Z10}
\end{align}
uniformly in $L \ge1$, $N \in \N$, 
and 
 $\kk \in \{-1,+1\}$.
Similarly, from 
Proposition \ref{PROP:psi}, we have 
\begin{align}
\Big\| \|\Psi_{L}(t)\|_{\CC^{-\dl}_{\ld, M}(\R^2)} \Big\|_{L^{p}(\PP)} \les 1, 
\label{Z11}
\end{align}
uniformly in $L\ge 1$ and $t = 0, T$.

Proceeding as in the proof of Theorem \ref{THM:main} 
(see \eqref{Y1} and \eqref{Y2}) 
with Proposition \ref{PROP:key2}, 
\eqref{A5}, \eqref{Z10}, and \eqref{Z11}, 
we obtain
\begin{align}
\sup_{L\ge 1} \sup_{N \in \N} \E \big[ \| u_{L, N}(0) \|_{\CC^{-\dl}_{\ld, M}(\R^2)}^{p} \big]  < \infty,
\label{Z12}
\end{align}
for any  $T\gg 1$ such that $K_1 e^{-T} \ll1 $,
where $K_1$ is as in 
Proposition \ref{PROP:key2}.

Now, let $\rho_L^\be$ be the weak limit of 
(of a $L$-dependent subsequence of) $\rho_{L, N}^\be$
constructed in Theorem \ref{THM:main}.
Then, 
it follows from \eqref{Z12} that 
\begin{align*}
\sup_{L\ge 1}  \E \big[ \| u_{L}(0) \|_{\CC^{-\dl}_{\ld, M}(\R^2)}^{p} \big]  < \infty.
\end{align*}
Recall that the choice of small $\delta>0$ was arbitrary.
Moreover, given $T,  p \gg1$, 
by choosing $M = M(T) \gg 1$, 
we see that 
 the conditions \eqref{para2} and \eqref{para}
reduces to
$M^{\frac 3p} \ll e^\ld \ll T^{-1} M$.
Namely, the range of admissible 
$\ld = \ld (T, p, M)$ is open.
Hence, it follows 
Lemma \ref{LEM:cpt}.
that the family $\{\rho_{L}^\be\}_{L\ge 1}$
is tight on $\CC^{-\dl}_{\ld, M}(\R^2)$, 
and hence  Theorem~\ref{THM:main2} follows from Prokhorov's theorem.
\end{proof}

%\begin{remark}\rm
%It is possible to modify the argument above to construct the 
%{\Rd massless sine-Gordon measure} since the RHS Schauder estimate in $\R^2$ from Lemma \ref{LEM:Schauder} decays in time. {\bf TODO: to discuss further}
%
%\end{remark}

\appendix

\section{Hyperbolic sine-Gordon model}
\label{SEC:wave}

In this appendix, 
%by adapting the proof of Proposition~\ref{PROP:key}, 
we prove pathwise global well-posedness of 
the 
hyperbolic sine-Gordon model \eqref{WW1} on 
$\T^2 = (\R/2\pi \Z)^2$
(with deterministic initial data) 
for $0 < \be^2 < 2\pi$.

We first go over the renormalization procedure for \eqref{WW1}
with deterministic initial data.
Let $\Psi^\wave$  be the solution to the linear stochastic damped wave equation:
\begin{align*}
\begin{cases}
\dt^2 \Psi^\wave + \dt\Psi^\wave +(1-\Dl)\Psi^\wave  = \sqrt{2}\xi\\
(\Psi^\wave,\dt\Psi^\wave)|_{t=0}=(0,0).
\end{cases}
%\label{WPsi1}
\end{align*}
Then, we have 
\begin{align} 
\Psi^\wave (t) 
&  = 
 \sqrt{2}\int_0^t\D(t - t')\xi(dt'), 
\label{WPsi2}
\end{align}
where 
$\D(t)$ denotes 
 the linear damped wave propagator defined by 
\begin{equation*}
\D(t) = e^{-\frac{t}2}\frac{\sin\Big(t\sqrt{\frac34-\Dl}\Big)}{\sqrt{\frac34-\Dl}}.
\end{equation*} 
Given $N \in \N$, let $\Psi_N^\wave = \Pii_N \Psi^\wave$, 
where $\Pii_N$ is as in \eqref{chi}.
Then, 
for $t \ge 0$, 
a direct computation with \eqref{WPsi2} 
 yields
\begin{align*}
\s_N^\wave(t)
& = \E\big[\big(\Psi_N^\wave(t, x)\big)^2\big]
 = \frac 1{2\pi^2} \sum_{n \in \Z^2}
\frac{\chi_N^2(n)}{\jbb{n}^2} \int_0^t e^{-(t- t')} \sin^2(t-t')\jbb{n}) dt'\\
& = \frac 1{4\pi^2} \sum_{n \in \Z^2}
\frac{\chi_N^2(n)}{\jb{n}^2}(1- e^{-t})
+ G_N (t) \\
&   = \frac {1- e^{-t}}{2\pi}\log N 
+ G_N(t),  
\end{align*}
where $\jbb{n} = \sqrt{\frac 34 + |n|^2}$.
Here, $G_N(t)$ is defined by 
\begin{align*}
G_N(t) = 
\frac {e^{-t} }{4\pi^2} \sum_{n \in \Z^2}
\frac{\chi_N^2(n)}
{\jbb{n}^2}
\bigg(\frac{1}{4\jb{n}^2}
+ \frac{1}{4\jb{n}^2} \cos(2t \jbb{n})
- \frac{\jbb{n}}{2\jb{n}^2} \sin(2t \jbb{n})\bigg), 
\end{align*}
which is uniformly bounded in $N \in \N$ and $t \ge 0$.
As in \eqref{CN}, we then define $\g_N^\wave(t) = \g_N^\wave(t; \be)$
by 
 \begin{equation}
\g_N^\wave (t) = e^{\frac{\be^2}{2}\s_N^\wave(t)}
\les  
N^{(1- e^{-t})\frac{\be^2}{4\pi}} \too \infty,  
\label{WCN}
 \end{equation}
as $N \to \infty$.

The discussion above leads us to the following renormalized 
truncated hyperbolic sine-Gordon model:
\begin{align}
\begin{cases}
\dt^2 u_N + \dt u_N + (1- \Dl)  u_N   +  \g_N^\wave \sin(\be \Pii_N u) = \sqrt{2}\Pii_N \xi\\
(u_N, \dt u_N) |_{t = 0} = (u_0, u_1), 
\end{cases}
\label{WW3}
\end{align}
where the renormalization constant $\g_N^\wave$ is now time-dependent.
Our main interest in this appendix is 
to study pathwise well-posedness, 
and not to construct invariant dynamics.
For this reason, 
we do not place  the frequency projector $\Pii_N$ on the nonlinearity, 
while 
the noise is truncated in frequencies.
Compare this with  \eqref{rSSG1}.
See also Remark~\ref{REM:wave2}.

As in the parabolic case, 
we  proceed with the  first order expansion: 
\begin{align*}
u_N = \Psi_N^\wave + v_N, 
\end{align*}
where $\Psi_N^\wave = \Pii_N \Psi^\wave$.
Then, the remainder term $v_N = u_N - \Psi_N^\wave$
satisfies
\begin{align}
\begin{cases}
 \dt^2 v_N +  \dt v_N  + (1-\Dl)  v_N 
+ \Im (f( v_N) \Ta_N^\wave)= 0\\
%\g_N^\wave  
% \Im  \sin \big(\be( v_N + \Psi_N^\wave  )\big) = 0 \\
(v_N, \dt v_N) |_{t = 0} = (u_0, u_1), 
\end{cases}
\label{WW4}
\end{align}
where $f(v) = e^{i \be v}$
and $\Theta_N^\wave$ is the  imaginary Gaussian multiplicative chaos
associated with the damped wave equation, 
 defined by
\begin{align*}
\Ta_N^\wave  = \g_N^\wave  e^{ i \be \Psi_N^\wave}
= e^{\frac{\be^2}{2}\s_N^\wave}
e^{ i \be \Psi_N^\wave}.
\end{align*}
In view of \eqref{WCN}, 
the (limiting) regularity property of 
$\Ta_N^\wave$ is exactly the same as that 
for $\Ta_N$ in the parabolic case
(as least in the range $0 < \be^2 < 4\pi$)
and thus Proposition~\ref{PROP:igmc}
also holds for $\Ta_N^\wave$.
See \cite[Lemma 2.2]{ORSW}
for a proof in the current hyperbolic case.

Fix $0 < \be^2 < 2\pi$
and small $\dl > 0$. %$\al > \frac{\be^2}{4\pi}$.
Let $\Ta^\wave$
 the limit of $\Ta_N^\wave$
in 
$L^p(\O;C(\R_+;  W^{- \frac{\be^2}{4\pi} - \dl, \infty}(\T^2)))$, 
endowed with the 
topology of uniform convergence on compact sets in time.
Then, by formally taking $N \to \infty$ in \eqref{WW4}, we have
\begin{align}
\begin{cases}
 \dt^2 v +  \dt v  + (1-\Dl)  v
+ \Im (f( v) \Ta^\wave)= 0\\
(v, \dt v) |_{t = 0} = (u_0, u_1). %\in \H^s(\T^2).
\end{cases}
\label{WW4a}
\end{align}

The following  theorem establishes 
pathwise global well-posedness of \eqref{WW4a}.

\begin{theorem}\label{THM:GWP}
Fix $0 < \be^2 < 2\pi$ and let $\dl > 0$ sufficiently small such that 
$\frac{\be^2}{4\pi} + \dl < \frac 12$.
Let $s = 1- \frac{\be^2}{4\pi} - \dl $.
Then, given 
 $(u_0, u_1) \in \H^s(\T^2)
: = H^s(\T^2) \times H^{s-1}(\T^2)$, 
there exists almost surely a unique global-in-time solution 
$(v, \dt v)$ to \eqref{WW4a}
in the class 
$C(\R_+;\H^{s}(\T^2))$.
%$C(\R_+;H^{1-\al}(\T^2))\cap C^1(\R_+;H^{-\al}(\T^2))$. 

\end{theorem}

\begin{remark}
By taking $N \to \infty$ in \eqref{WW3}, 
we arrive at
 the (formal) limiting equation:
\begin{align*}
\begin{cases}
\dt^2 u + \dt u + (1- \Dl)  u   +  \infty \cdot \sin(\be  u) = \sqrt{2} \xi\\
(u, \dt u) |_{t = 0} = (u_0, u_1).
\end{cases}
\end{align*}
By pathwise global well-posedness
of the hyperbolic sine-Gordon model, 
we mean that $u = \Psi^\wave + v$, 
where $\Psi^\wave$ is as in  \eqref{WPsi2}
and $v$ is a pathwise global solution to \eqref{WW4a}.
\end{remark}

%
%\begin{proposition}\label{PROP:GWP2}
%\end{proposition}

%
%
%By writing \eqref{WW4a} in the Duhamel formulation, we have  
%\begin{align}
%v_N = S(t)(u_0, u_1) 
%-   \I^\wave  \big(  \Im (f( v_N) \Ta_N^\wave)\big), 
%\label{WW5}
%\end{align} 
%

In order to  prove pathwise global well-posedness
of \eqref{WW4a}, 
we consider  the following deterministic damped nonlinear wave equation on~$\T^2$:
\begin{align}
\begin{cases}
\dt^2 v + \dt v + (1- \Dl) v = \sum_{\kk \in \{+, -\}} f^\kk(v) \Ta^\kk \\
(v, \dt v) |_{t = 0} = (u_0, u_1),  %\in \H^s(\T^2)
\end{cases}
\label{WW9}
\end{align}
where $f^\pm(v) = e^{\pm i \be v}$
and 
\[ \Ta^\kk\in 
L^\infty_\text{loc}(\R_+;W^{-\al,\infty}(\T^2)), \quad \kk \in \{+, -\}\]
for some  
 $\al > 0$. 
A slight modification of the proof of \cite[Proposition 3.1]{ORSW}
yields the following pathwise local well-posedness result  for~\eqref{WW9}.

\begin{lemma}\label{LEM:LWP}
Given $0<\al <\frac12$, 
let $(u_0, u_1) \in \H^{1-\al}(\T^2)$
 and 
 $\U^\kk$, $\kk \in \{+, -\}$,  be distributions in $L^\infty([0,1];W^{-\al,\infty}(\T^2))$. % for any $p\ge 1$. 
 Then,  there exist
 \[T = T\Big(\| (u_0, u_1)\|_{\H^{1-\al}}, 
 \|\Ta^+\|_{L^\infty([0,1];W^{-\al,\infty}_x)},  \|\Ta^-\|_{L^\infty([0,1];W^{-\al,\infty}_x)}\Big)> 0\]
 and a unique solution $v$ to \eqref{WW9} 
 on the interval $[0, T]$, belonging to  the class\textup{:}
 \begin{align*}
 Y^{1-\al}(T): =  C([0,T];H^{1-\al}(\T^2))\cap C^1([0,T];H^{-\al}(\T^2)). %\cap L^{\infty}([0,T];L^{\frac2\al}(\T^2)).
 \end{align*} 
 Moreover, the solution map\textup{:} $(u_0, u_1, \Ta^+, \Ta^-) \mapsto v$ is continuous.
\end{lemma}

Suppose that, given any $T > 0$, 
$\Ta^\kk \in L^\infty([0,T];W^{-\al,\infty}(\T^2))$, $\kk \in \{+, -\}$. 
Then, it follows from Lemma~\ref{LEM:LWP}
that \eqref{WW9} is globally well-posed, 
if 
\begin{align}
 \sup_{0 \le t \le T}
\|(v(t), \dt v(t))\|_{\H^{1- \al}} < \infty
\label{ZZ1}
\end{align}
for each finite $T> 0$.
In the following, we show that \eqref{ZZ1} indeed holds
for $0<\al <\frac12$, 
from which we easily deduce Theorem \ref{THM:GWP}
by setting 
$\al = \frac{\be^2}{4\pi} + \dl$. % (with sufficiently small $\dl > 0$).
% provided that 
% $0 < \be^2 < 2\pi$.

Before proceeding further, 
we state a preliminary lemma.

\begin{lemma}\label{LEM:toolbox}
\begin{enumerate}
\item Given $s \in \R$, 
let $u$ be a solution to  the linear damped wave equation on $\R_+ \times \T^2$\textup{:}
\begin{align}
\begin{cases}
\dt^2 u +\dt u + (1-\Dl) u  = f\\
(u,\dt u)|_{t=0}=(u_0,u_1) \in \H^s(\T^2).
\end{cases}
\label{WZ1}
\end{align}
Then,  
for any $T > 0$, 
we have
\begin{align}
\|u\|_{C_TH^{s}_x} + \|\dt u\|_{C_TH^{s-1}_x} 
 \les \|(u_0,u_1)\|_{\H^s} + \|f\|_{L^1_TH^{s-1}_x}.
\label{WZ2}
\end{align}

\item\textup{(fractional Leibniz rule).}
Let $s> 0$ and 
 $1<p_j,q_j,r\le \infty$ with $\frac1{p_j}+\frac1{q_j}=\frac1r$, $j=1,2$. Then, we have 
\begin{align*}
\big\|\jb{\nabla}^s(fg)\big\|_{L^r(\T^2)} 
\les \big\|\jb{\nabla}^s f\big\|_{L^{p_1}(\T^2)}\|g\|_{L^{q_1}(\T^2)} 
+ \|f\|_{L^{p_2}(\T^2)}\big\|\jb{\nabla}^s g\big\|_{L^{q_2}(\T^2)}.
\end{align*}

\item
Let $s > 0$, $1 < p \le \infty$, and 
 $1<q,r<\infty$ such that $\frac1p+\frac1q\le \frac1r + \frac{s}2$
 and $q, r' \ge p'$. Then, we have
\begin{align*}
\big\|\jb{\nabla}^{-s}(fg)\big\|_{L^r(\T^2)} \les \big\|\jb{\nabla}^{-s}f\big\|_{L^p(\T^2)}
\big\|\jb{\nabla}^s g\big\|_{L^q(\T^2)}.
\end{align*}

\item \textup{(fractional chain rule).}
Let $F$ be a Lipschitz function on $\R$ with Lipschitz constant $L$. Then, given $0 < s <1$ and $1 < p < \infty$, we have
\begin{align*} 
\big\| |\nb|^s F(f) \big\|_{L^p (\T^2)} \les L  \big\| |\nb|^s f \big\|_{L^p (\T^2)}
\end{align*}
for any $f \in C^\infty(\T^2)$.
\end{enumerate}
\end{lemma}

The energy estimate \eqref{WZ2} is obvious from 
writing \eqref{WZ1} in the Duhamel formulation.
See~\cite{BOZ} for (ii) and (iii);
see also \cite{GKO}.
As for (iv), see \cite{Gatto}.\footnote{The fractional chain rule on $\R^d$ was essentially proved in \cite{CW}. As pointed out in \cite{Staffilani}, however, 
the proof in \cite{CW} needs
a small correction, which  yields the fractional chain  rule in a 
less general context.}

We now prove \eqref{ZZ1}.
Given $T \ge 1$, 
let $v \in Y^{1-\al}(T)$ be a solution to \eqref{WW9} on the time interval $[0, T]$.
Then, 
by 
Lemma~\ref{LEM:toolbox}, 
we have
\begin{align}
\begin{split}
\| v \|_ {Y^{1-\al}(T)} 
&\les \|(u_0, u_1) \|_{\H^{1-\al}}
+ \sum_{\kk \in \{+, -\}} \| f^\kk(v) \Ta^\kk  \|_{L^1_TH^{-\al}_x}\\
& \les
\|(u_0, u_1) \|_{\H^{1-\al}}
+   T
\sum_{\kk \in \{+, -\}}
  \|f^\kk(v) \|_{L^{\infty}_TH^\al_x}\|\Ta^\kk\|_{C_TW^{-\al,\infty}_x}.
\end{split}
\label{ZZ6}
\end{align}
On the other hand, 
by interpolation, Lemma~\ref{LEM:toolbox}\,(iv), 
and the boundedness of $f^\kk(v)$, 
we have 
\begin{align}
\begin{split}
\|f^\kk(v)\|_{H^\al } 
& \les \| f^\kk(v)\|_{H^{(1-\ta)^{-1}\al }}^{1-\ta} \|f^\kk(v)\|_{L^2}^\ta \\
& \les \| v\|_{H^{(1-\ta)^{-1}\al }}^{1-\ta}
 \le \| v\|_{H^{1-\al}}^{1-\ta}
\end{split}
\label{ZZ7}
\end{align}
for $\kk \in \{+, -\}$, provided that  $\ta >0$ is sufficiently small such that 
\begin{align*}
(1-\ta)^{-1}\al  < {1-\al}, 
\end{align*}
which is possible since $0 <  \al <  \frac 12$.
Hence, from \eqref{ZZ6}, \eqref{ZZ7}, and Young's inequality, 
we obtain
\begin{align*}
\|v\|_{Y^{1-\al}(T)}
&  \le C_1 \|(u_0, u_1) \|_{\H^{1-\al}}
 + 
C_1 T
\sum_{\kk \in \{+, -\}}
\|\Ta^\kk\|_{C_TW^{-\al,\infty}_x}
\|v\|_{Y^{1-\al}(T)}^{1-\ta}
\\
&  \le C_1 \|(u_0, u_1) \|_{\H^{1-\al}}
+ C_2(T)  \sum_{\kk \in \{+, -\}} \| \Ta^\kk \|_{C_T W_x^{-\al, \infty}}^\frac1 \ta + \frac{1}{2} 
\|v\|_{Y^{1-\al}(T)}
\end{align*}
for some $C_1, C_2(T) > 0$.
This proves \eqref{ZZ1}
(and hence Theorem \ref{THM:GWP}).

\begin{remark}\label{REM:wave2}
It is natural to wonder if we can use 
the hyperbolic sine-Gordon model
to 
 implement a PDE construction of the sine-Gordon measure $\rho^\be$ 
 as in the proof of Theorem~\ref{THM:main}.
 Unfortunately, 
due to  the lack of smoothing
under the homogeneous linear damped wave propagator, 
 this is not possible  with the techniques presented in this note.
\end{remark}

\begin{ackno}
T.O.\ was supported by the European Research Council (grant no.~864138 ``SingStochDispDyn")
and  by the EPSRC 
Mathematical Sciences
Small Grant  (grant no.~EP/Y033507/1).
Y.Z.~was   funded by the chair of probability and PDEs at EPFL. 
\end{ackno}


\begin{thebibliography}{99}


\bibitem{AK}
S.~Albeverio, S.~Kusuoka,
{\it The invariant measure and the flow associated to the $\Phi^4_3$-quantum field model},
Ann. Sc. Norm. Super. Pisa Cl. Sci.  20 (2020), no. 4, 1359--1427.

\bibitem{AK2}
S.~Albeverio, S.~Kusuoka,
{\it Construction of a non-Gaussian and rotation-invariant $\Phi^4$-measure and associated flow on $\R^3$ through stochastic quantization},
arXiv:2102.08040 [math.PR].

\bibitem{AS}
N.~Aronszajn, K.~Smith, 
{\it Theory of Bessel potentials. I,}
 Ann. Inst. Fourier (Grenoble) 11 (1961),  385--475. 


%\bibitem{BCD}
%H.~Bahouri, J.-Y.~Chemin, 
%R.~Danchin, 
%{\it Fourier analysis and nonlinear partial differential equations,}
%Grundlehren der Mathematischen Wissenschaften [Fundamental Principles of Mathematical Sciences], 
%343. Springer, Heidelberg, 2011. xvi+523 pp. 







\bibitem{Bara}
N.~Barashkov,
{\it 
A stochastic control approach to Sine Gordon EQFT}, 
arXiv:2203.06626 [math.PR].


\bibitem{BG}
N.~Barashkov, M.~Gubinelli, 
{\it  A variational method for $\Phi^4_3$}, 
Duke Math. J.
 169 (2020), no. 17, 3339--3415.
%

%\bibitem{Bass}
%R.~Bass, 
%{\it Stochastic processes.}
%Cambridge Series in Statistical and Probabilistic Mathematics, 33. Cambridge University Press, Cambridge, 2011. xvi+390 pp. 


\bibitem{BB}
R.~Bauerschmidt, T.~Bodineau, 
{\it Log-Sobolev inequality for the continuum sine-Gordon model},
 Comm. Pure Appl. Math. 74 (2021), no. 10, 2064--2113. 

\bibitem{BOZ}
\'A.~B\'enyi, T.~Oh, T.~Zhao,
{\it Fractional Leibniz rule on the torus,}
Proc. Amer. Math. Soc.
 (2024). https://doi.org/10.1090/proc/17007 


\bibitem{Bony}
J.-M.~Bony, 
{\it Calcul symbolique et propagation des singularit\'es pour les \'equations aux d\'eriv\'ees partielles non lin\'eaires,}
Ann. Sci. \'Ecole Norm. Sup.  14 (1981), no. 2, 209--246.


\bibitem{BO94}
J.~Bourgain,
{\it Periodic nonlinear Schr\"odinger equation and invariant measures},
Comm. Math. Phys. 166 (1994), no. 1, 1--26.


\bibitem{BO96}
J.~Bourgain, 
{\it Invariant measures for the 2D-defocusing nonlinear Schr\"odinger equation}, 
Comm. Math. Phys. 176 (1996), no. 2, 421--445. 



\bibitem{BC}
B.~Bringmann, S.~Cao,
{\it 
Global well-posedness of the dynamical sine-Gordon model up to $6\pi$}, 
arXiv:2410.15493 [math.AP].




\bibitem{CFW}
A.~Chandra, G.L.~Feltes, H.~Weber,
{\it A priori bounds for 2-d generalised parabolic Anderson model},
arXiv:2402.05544 [math.AP].

\bibitem{CHS}
A.~Chandra, M.~Hairer, H.~Shen,
{\it The dynamical sine-Gordon model in the full subcritical regime},
arXiv:1808.02594 [math.PR].


\bibitem{CW}
M.~Christ, M.~Weinstein,
{\it Dispersion of small amplitude solutions of the generalized Korteweg-de Vries equation.}
J. Funct. Anal.
100 (1991), 87--109.





\bibitem{DPD}
G.~Da Prato, A.~Debussche, 
{\it Strong solutions to the stochastic quantization equations}, 
Ann. Probab. 31 (2003), no. 4, 1900--1916.



\bibitem{DH1}
J.~Dimock, T.R.~Hurd,
{\it A renormalization group analysis of the Kosterlitz-Thouless phase},
Comm. Math. Phys. 137 (1991), no. 2, 263--287.

\bibitem{DH2}
J.~Dimock, T.R.~Hurd,
{\it Construction of the two-dimensional sine-Gordon model for $\be < 8\pi$},
Comm. Math. Phys. 156 (1993), no. 3, 547--580.

\bibitem{DH3}
J.~Dimock, T.R.~Hurd,
{\it Sine-Gordon revisited},
Ann. Henri Poincar\'e 1 (2000), no. 3, 499--541.



\bibitem{Fro}
J.~Fr\"ohlich, 
{\it Quantized ``sine-Gordon'' equation with a nonvanishing mass term in two space-time dimensions},
 Phys. Rev. Lett. 34 (1975), 833--836. 

\bibitem{FP}
J.~Fr\"ohlich, Y.M.~Park, 
{\it Remarks on exponential interactions and the quantum sine-Gordon equation in two space-time dimensions},
Helv. Phys. Acta 50 (1977), no. 3, 315--329. 

\bibitem{Gatto}
A.E.~Gatto,
{\it Product rule and chain rule estimates for fractional derivatives on spaces that satisfy the doubling condition},
J. Funct. Anal. 188 (2002), no. 1, 27–37.

\bibitem{GF}
J.~Glimm, A.~Jaffe,
{\it Quantum physics. A functional integral point of view},
Second edition. Springer-Verlag, New York, 1987. xxii+535 pp.
%


\bibitem{Gra1}
L.~Grafakos, 
{\it Classical Fourier analysis}, Third edition. Graduate Texts in Mathematics, 249. Springer, New York, 2014. xviii+638 pp. 



\bibitem{Gra2}
L.~Grafakos, 
{\it Modern Fourier analysis}, Third edition. Graduate Texts in Mathematics, 250. Springer, New York, 2014. xvi+624 pp.


\bibitem{GIP}
M.~Gubinelli, P.~Imkeller, N.~Perkowski, 
{\it Paracontrolled distributions and singular PDEs,} Forum Math. Pi 3 (2015), e6, 75 pp. 


\bibitem{GH2}
M.~Gubinelli, M.~Hofmanov\'a, 
{\it A PDE construction of the Euclidean $\phi^4_3$ quantum field theory},
Comm. Math. Phys. 384 (2021), no. 1, 1–75.



\bibitem{GKO}
M.~Gubinelli, H.~Koch, T.~Oh,
{\it  Renormalization of the two-dimensional stochastic nonlinear wave equations,}
 Trans. Amer. Math. Soc.
 370 (2018), no 10, 7335--7359.


%\bibitem{GKO2}
%M.~Gubinelli, H.~Koch, T.~Oh,
%{\it Paracontrolled approach to the three-dimensional stochastic nonlinear wave equation with quadratic nonlinearity},  J. Eur. Math. Soc. 26 (2024), no. 3, 817--874. 
 
%\bibitem{GKOT}
%M.~Gubinelli, H.~Koch, T.~Oh, L.~Tolomeo,
%{\it Global dynamics for  the two-dimensional stochastic nonlinear wave equations,}
%Int. Math. Res. Not. 2022, no. 21, 16954--16999. 
%

\bibitem{GM}
M.~Gubinelli, S.-J.~Meyer,
{\it The FBSDE approach to sine-Gordon up to $6\pi$},
arXiv:2401.13648 [math-ph].


\bibitem{Hairer}
M.~Hairer, 
{\it A theory of regularity structures,} Invent. Math. 198 (2014), no. 2, 269--504. 

\bibitem{HRW}
M.~Hairer, M.~D.~Ryser, H.~Weber,
{\it Triviality of the 2D stochastic Allen-Cahn equation},
Electron. J. Probab. 17 (2012), no. 39, 14 pp. 


\bibitem{HS}
M.~Hairer, H.~Shen,
{\it The dynamical sine-Gordon model}, Comm. Math. Phys. 341 (2016), no. 3, 933--989. 


%\bibitem{Kato}
%T.~Kato,
%{\it On nonlinear Schr\"odinger equations. II. $H^s$-%solutions and unconditional well-posedness},
%J. Anal. Math. 67 (1995), 281--306. 


\bibitem{McKean}
H.P.~McKean, 
{\it Statistical mechanics of nonlinear wave equations. IV. Cubic Schr\"odinger}, Comm. Math.
Phys. 168 (1995), no. 3, 479--491. 
{\it Erratum: Statistical mechanics of nonlinear wave equations. IV. Cubic
Schrodinger}, Comm. Math. Phys. 173 (1995), no. 3, 675.




\bibitem{MW2}
J.-C.~Mourrat, H.~Weber,
{\it Global well-posedness of the dynamic $\Phi^4$ model in the plane},
Ann. Probab.45 (2017), no. 4, 2398–2476.



\bibitem{MW}
J.-C.~Mourrat, H.~Weber,
{\it The dynamic $\Phi^4_3$ model comes down from infinity},
Comm. Math. Phys. 356 (2017), no. 3, 673--753.




\bibitem{PW}
G.~Parisi, Y.S.~Wu, 
{Perturbation theory without gauge fixing},
Sci. Sinica 24 (1981), no. 4, 483--496.



%\bibitem{ORSW1}
%T.~Oh, T.~Robert, P.~Sosoe, Y.~Wang,
%{\it On the two-dimensional hyperbolic stochastic sine-Gordon equation},
% Stoch. Partial Differ. Equ. Anal. Comput. 9 (2021), 1--32. 

\bibitem{ORSW}
T.~Oh, T.~Robert, P.~Sosoe, Y.~Wang,
{\it Invariant Gibbs dynamics for the dynamical sine-Gordon model},
Proc. Roy. Soc. Edinburgh Sect. A 151 (2021), no. 5, 1450-1466.

\bibitem{ORW}
T.~Oh, T.~Robert, Y.~Wang,
{\it  On the parabolic and hyperbolic Liouville equations},
 Comm. Math. Phys. 387 (2021), no. 3 1281--1351. 


\bibitem{OTWZ}
T.~Oh, L.~Tolomeo, Y.~Wang, G.~Zheng,
{\it  Hyperbolic $P(\Phi)_2$-model on the plane},
arXiv:2211.03735 [math.AP].



\bibitem{OZ}
T.~Oh, Y.~Zine,
{\it  Hyperbolic sine-Gordon model beyond the first threshold},
preprint.

\bibitem{RSS}
S.~Ryang, T.~Saito, K.~Shigemoto,
{\it Canonical stochastic quantization},
Progr. Theoret. Phys. 73 (1985), no. 5, 1295--1298. 




\bibitem{SZZ}
H.~Shen, R.~Zhu, X.~Zhu
{\it Global well-posedness for 2D generalized parabolic Anderson model via paracontrolled calculus}, 
arXiv:2402.19137 [math.AP].

%\bibitem{Simon}
%B.~Simon, 
%{\it  The $P(\varphi)_2$ Euclidean (quantum) field theory,} %Princeton Series in Physics. Princeton University Press, Princeton, N.J., 1974. xx+392 pp.

\bibitem{Staffilani}
G.~Staffilani,
{\it The initial value problem for some dispersive differential equations},
Thesis (Ph.D.) -- The University of Chicago. 1995. 88 pp.

%\bibitem{Taylor}
%M.~Taylor,
%{\it Tools for PDE}, 
%Pseudodifferential operators, paradifferential operators, and layer potentials. Mathematical Surveys and Monographs, 81. American Mathematical Society, Providence, RI, 2000. x+257 pp.



\bibitem{Zine}
Y.~Zine, 
{\it Smoluchowski-Kramers approximation for singular stochastic wave equations in two dimensions},
arXiv:2206.08717 [math.AP].


\end{thebibliography}
\end{document}